\documentclass{amsart}
\usepackage{amsaddr}
\title{Minimizing Closed Geodesics on Polygons and Disks}
\author{Ian Adelstein, Arthur Azvolinsky, Joshua Hinman, Alexander Schlesinger}
\date{}
\address{Department of Mathematics, Yale University \\ New Haven, CT 06520 United States}
\usepackage{geometry}
\usepackage{listings}
\usepackage{comment}
\usepackage{amsmath, amsthm, amssymb}
\usepackage{breqn}
\usepackage{tikz}
\usepackage{subcaption, graphicx}
\usetikzlibrary{arrows}
\usetikzlibrary{chains,fit,shapes}

\newtheorem{theo}{Theorem}[section]
\newtheorem{defi}[theo]{Definition}
\newtheorem{rema}[theo]{Remark}
\newtheorem{form}[theo]{Formula}
\newtheorem{prop}[theo]{Proposition}
\newtheorem{lemma}[theo]{Lemma}
\newtheorem{conjecture}[theo]{Conjecture}
\newtheorem{Corollary}[theo]{Corollary}
\newtheorem{example}[theo]{Example}
\newcommand{\mychi}{\raisebox{0pt}[1ex][1ex]{$\chi$}}

\begin{document}
\begin{abstract}
In this paper we study 1/k geodesics, those closed geodesics that minimize on all subintervals of length $L/k$, where $L$ is the length of the geodesic. We develop new techniques to study the minimizing properties of these curves on doubled polygons, and demonstrate a sequence of doubled polygons whose closed geodesics exhibit unbounded minimizing properties. We also compute the length of the shortest closed geodesic on doubled odd-gons and show that this length approaches $4\cdot$diameter.
\end{abstract}
\maketitle
\section{Introduction}
A geodesic in a metric space $\mathcal{M}$ is a locally length-minimizing curve.  Some metric spaces admit closed geodesics, $\gamma \colon S^1 \to \mathcal{M}$, which can be viewed as maps from the circle $S^1=[0,2\pi]$ into the metric space. In \cite{sor}, Sormani defined a $\textit{1/k geodesic}$ to be a constant speed closed geodesic such that $d(\gamma(t), \gamma(t+2\pi/k)) = L/k$ for every $t \in S^1$, where $L$ is the length the geodesic. Intuitively, a $1/k$ geodesic realizes the distance between points which are $1/k$ of the geodesic apart. As a first example note that the great circles on the round sphere are $1/2$ geodesics. 

By compactness of $S^1$ and the local length minimization property, every closed geodesic is a $1/k$ geodesic for some integer $k\geq 2$. It is clear that if $\gamma$ is a $1/k_0$ geodesic, then it is also a $1/k$ geodesic for all $k \geq k_0$. In order to specify the optimal value of $k$ for a given closed geodesic, we use the notion of $\textit{minimizing index}$, defined by Sormani, where $minind(\gamma)$ is the smallest integer $k \geq2$ such that $\gamma$ is a $1/k$ geodesic. The concept of minimizing index extends to a metric space $\mathcal{M}$, by defining $minind(\mathcal{M})$ to be the smallest integer $k \geq2$ such that the space admits a $1/k$ geodesic.

In this paper we study $minind(X_n)$, where $X_n$ is the doubled regular $n$-gon:~the metric space obtained by gluing two $n$-gons along their common boundary. We also study the minimizing index of closed geodesics on the doubled disk, the Gromov-Hausdorff limit of the doubled regular $n$-gons, which we denote by $X_\infty$. On these spaces, closed geodesics can be viewed as the concatenation of line segments between edge points, alternating between the front and back faces of the space. 
We exploit the symmetry of the disk to prove our first result:

\begin{theo}
\label{double disk minind}
For all closed geodesics $\gamma \in X_\infty$, $minind(\gamma)$ is the number of segments of $\gamma$.
\end{theo}

The lowest possible minimizing index for any closed geodesic (and thus for any metric space) is two, as it is impossible for a segment longer than half the length of a closed geodesic to be minimizing: traversing the geodesic in the opposite direction always provides a shorter path. 
One can therefore ask whether a given metric space admits a half geodesic ($1/2$ geodesic).
It was shown in \cite[Proposition 2.1]{ade} that $X_{2n}$ admits precisely $n$ half geodesics, and that $X_{2n+1}$ fails to admit half geodesics.
The $n$ half geodesics on $X_{2n}$ are precisely the perpendicular bisectors of parallel edges, passing through the center of each face of the polygon. 

Given that $X_{2n+1}$ fails to admit a half geodesic it is natural to consider  the smallest integer $k>2$ such that $X_{2n+1}$ admits a 1/k geodesic, i.e.~$minind(X_{2n+1})$.
This is studied in \cite{fong} by first defining the $\textit{over-under}$ curve, a closed geodesic on $X_n$ that passes through the midpoints of adjacent edges. These geodesics look like inscribed $n$-gons within $X_n$, which alternate between the front and back of the polygon (see Figure \ref{overunder}). 
Due to this alternating property, if $n$ is odd, the inner $n$-gon is traversed twice before closing up smoothly. 
The main result \cite[Theorem 2.4]{fong} shows that the minimizing index of the over-under curve on $X_n$ (for $n$ odd) is equal to $2n$.

\begin{figure}
\includegraphics[scale=.45]{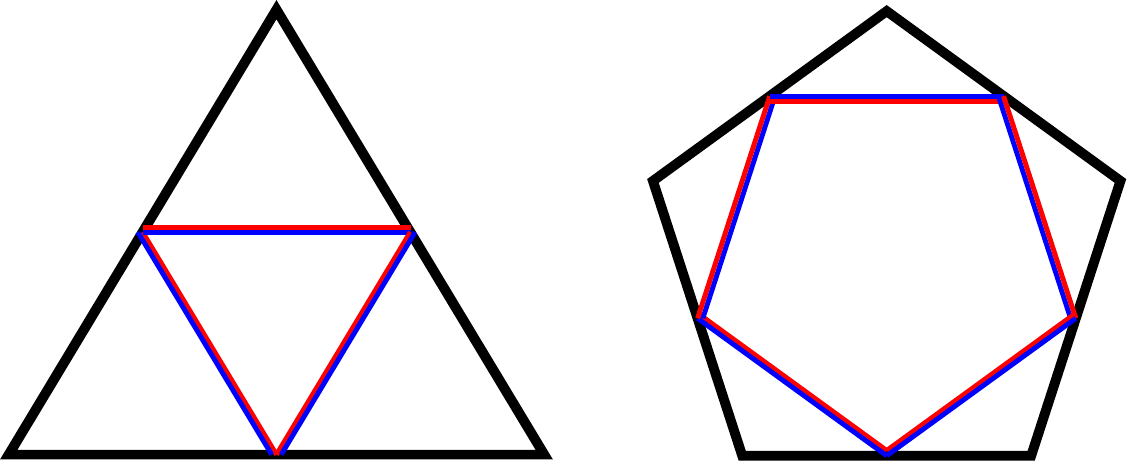}
\caption{The over-under curves on $X_3$ and $X_5$}
\label{overunder}
\end{figure}

\begin{figure}
\includegraphics[scale=.6]{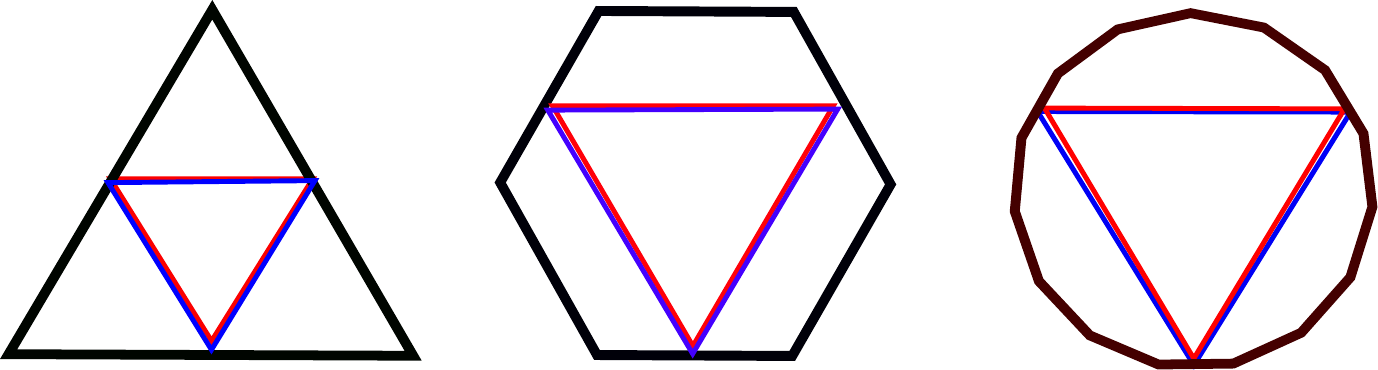}
\caption{The inscribed triangle geodesics on $X_{3k}$, where $k = 1,2,5$}
\label{3k}
\end{figure}

For $p$ an odd prime, the over-under curve gives an upper bound of $2p$ on the minimizing index of $X_p$. However, when $n$ is odd but not prime, we can write $n=kp$, for $p$ prime and $k$ odd. 
By connecting every $k$-th midpoint, we can inscribe a $p$-gon in $X_n$, again traversed twice before closing smoothly (see Figure $\ref{3k}$). 
We apply our Theorem~\ref{double disk minind} to show that this geodesic has minimizing index $2p$, thus extending \cite[Theorem 2.4]{fong} to the setting $n=kp$:
 
\begin{Corollary}
\label{divisor}
For all $n$, $minind(X_n) \leq 2d$, where $d$ is the smallest prime divisor of $n$.
\end{Corollary}

This corollary gives an upper bound on the minimizing index of the doubled regular $n$-gon, which is conjectured to be tight for odd $n$. Whether this bound is tight remains an open problem, but we are able to provide the following limiting lower bound:
\begin{theo}
\label{infinity}
For $p$ prime, as $p \to\infty$, the minimizing index of $X_p$ tends towards infinity. 
\end{theo}

%This statement has a natural generalization, which we present here:
%\begin{theo}
%\label{generalization}
%If $(n_i)$ is a sequence of natural numbers such that the greatest common divisor of any subsequence is 1, then $minind(X_{n_i}) \to \infty$.
%\end{theo}

Recall from \cite[Proposition 2.1]{ade} that the odd-gons fail to admit half geodesics.
The sequence of odd-gons converges in the Gromov-Hausdorff sense to the double disk, which does admit half geodesics (any pair of diameters meeting at the boundary circle).  This sequence shows that minimizing index need not converge in the Gromov-Hausdorff limit, but it was unknown previously how bad these differences could be.  Theorem \ref{infinity} shows that minimizing index can diverge in the worst possible way.
We state this result formally in the following corollary:
\begin{Corollary}
\label{coro}
Let $\mathcal{M}_i$ converge to $\mathcal{M}$ in the Gromov-Hausdorff sense.  The minimzing indices of the $\mathcal{M}_i$ need not converge to the minizing index of $\mathcal{M}$.  In particular, from Theorem \ref{infinity} we have that $minind(X_p) \rightarrow \infty$, whereas $$minind\left(\lim_{p\to\infty} X_p\right) = minind(X_\infty) = 2.$$
\end{Corollary}

Sormani initially defined $1/k$ geodesics to study convergence of the length spectra of sequences of metric spaces. 
She notes that the length spectrum of a metric space $\mathcal{M}$, defined to be the set of lengths of closed geodesics in $\mathcal{M}$, is not continuous under Gromov-Hausdorff convergence; even if $\mathcal{M}_i$ converge to $\mathcal{M}$ in the Gromov-Hausdorff sense, there could be sequences of geodesics in the $\mathcal{M}_i$ that do not converge to a geodesic on $\mathcal{M}$, c.f.~\cite[Example 7.2]{sor}.

This convergence issue is resolved if we instead consider the $1/k$ length spectrum, the set of lengths of $1/k$ geodesics. Sormani's main result \cite[Theorem 7.1]{sor} states that if the $\mathcal{M}_i$ converge to $\mathcal{M}$ in the Gromov-Hausdorff sense, then the $1/k$ length spectra of the $\mathcal{M}_i$ converge to a subset of the $1/k$ length spectrum of $\mathcal{M}$ union $\{0\}$.
An immediate corollary \cite[Corollary 7.2]{sor} is that if a sequence of geodesics $\gamma_i \in \mathcal{M}_i$ have lengths converging to a non-zero value that is not in the length spectrum of $\mathcal{M}$, then the minimizing indices of the $\gamma_i$ grow without bound. 

We show by example that the converse to this corollary does not hold. Details for this example are included in Section~\ref{vshapesec}, and an illustration can be found in Figure~\ref{vshape}.
\begin{example}
\label{vshapeEx}
There exists a sequence of V-shaped geodesics on the doubled regular odd-gons whose minimizing indices grow without bound, yet which converge to an iterated diameter geodesic on the double disk with minimizing index 4. 
\end{example}

\begin{figure}
\includegraphics[scale=0.6]{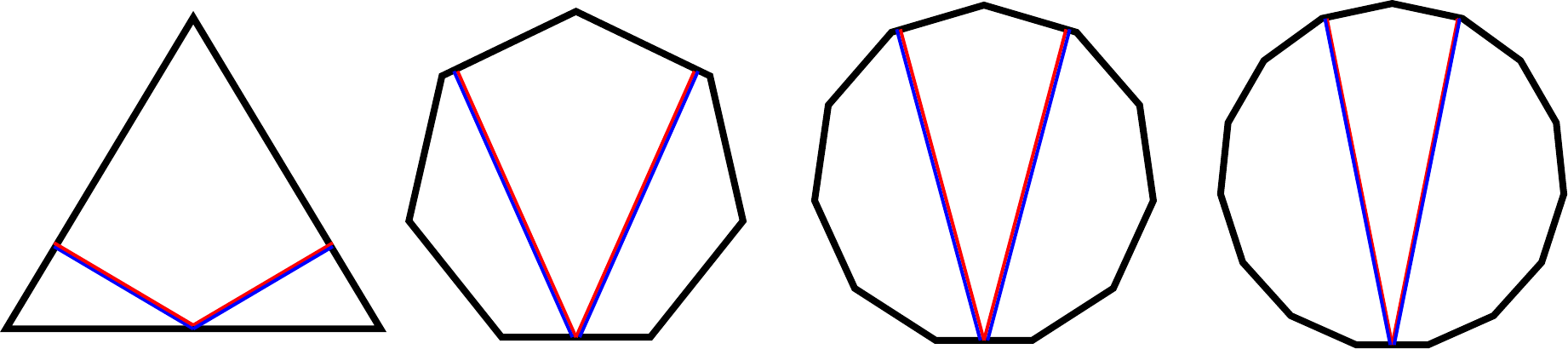}
\caption{A sequence of period 4 curves on $X_n$ for $n=3,7,11,15$}
\label{vshape}
\end{figure}

These V-shaped geodesics on the doubled odd-gons are of interest in terms of bounds on $L(S^2,g)$, the length of the shortest closed geodesic on the Riemannian 2-sphere. The best known bound in terms of the diameter is $L\leq 4$diam due to Nabutovsky and Rotman \cite{rotman}, and independently Sabourau \cite{sab}. The naive bound of $L\leq 2$diam does not hold as is demonstrated by a Zoll sphere with $L>2$diam \cite{croke}. The V-shaped geodesics have the following property:

\begin{theo}\label{vshapethm}
The V-shaped geodesics are the unique shortest closed geodesics on $X_{2n+1}$, and have length approaching $4diam(X_{2n+1})$. % and $\frac{8}{\pi} area(X_{2n+1})$.
\end{theo}

It would appear as though the doubled odd-gons realize the best known diameter upper bound for the length of shortest closed geodesic on the 2-sphere. However, the upper bound of $4$diam given in \cite{rotman} and \cite{sab} is only for smooth Riemannian metrics, whereas the doubled polygons are metric spaces with cone points (corners). One could try to extend the $4$diam bound to an appropriate setting, but it is already known that there exist doubled triangles with angles close to 30-60-90 which have shortest closed geodesic of arbitrary length \cite{schwartz}. One could hope to smooth the metric on the odd-gons to achieve an $L(S^2, g)$ close to $4$diam, but the naive smoothing collapses the curve into an iterate of a closed geodesic with length bounded above by $2$diam.

We now provide an overview of the paper.
%I'm noticing that we don't have any description of Section 2 in this overview, and also that Section 2 is incredibly short and doesn't have any significant results.  The background info contained in Section 2 is certainly important, but is it worth having its own section?
Section~\ref{disk} of this paper will be dedicated to proving Theorem~\ref{double disk minind}. We will provide a classification of the minimizing index for closed geodesics on the doubled disk using purely geometric arguments. Corollary~\ref{divisor} then shows that $minind(X_n) \leq 2d$, where $d$ is the smallest prime divisor of $n$.

In Section~\ref{limit} we prove Theorem~\ref{infinity}, providing a limiting lower bound on the minimizing index of $X_p$. In the proof we develop new ways to describe geodesic paths on the doubled n-gons:~the skip number, which tracks the number of vertices between consecutive edge points on a geodesic, and the vertex ratio, which measures how close to a vertex the geodesic hits an edge.
The sequences of skip numbers and vertex ratios across the path of a geodesic reveals surprising arithmetic properties, which are used to prove the theorem. Section~\ref{vshapesec} explores the V-shaped geodesics on the doubled odd-gons and their various properties. 

In addition to \cite{ade, fong, sor} see also \cite{ade2, epstein, wkh} for more on 1/k geodesics.

\section{General Properties of Geodesics on Doubled Polygons and Disks}\label{general}
An important property of geodesics on $X_n$ is that the angle of incidence equals the angle of reflection whenever the geodesic intersects an edge of the doubled polygon.
This implies that the geodesics we study are essentially signed billiards paths on regular $n$-gons, c.f.~\cite{veech}.
In particular, our geodesics are composed of line segments from edge to edge. Based on this decomposition, we define the following useful property of these geodesics:
\begin{defi}
The $\textit{period}$ of a closed geodesic $\gamma$, denoted $per(\gamma)$, is the number of line segments it contains.
\end{defi}

Note that we can never have two consecutive segments on the same face of the doubled polygon. Such a curve would fail to be locally length minimizing at the edge point.
Indeed, in any neighborhood around such an edge point, there would be a point on each segment, and a straight-line path would be shorter than the path through the edge due to the triangle inequality.
Thus the segments must alternate faces and we conclude that the period of a closed geodesic must be even.

As an example, recall from the introduction that the over-under curves on $X_p$ have period $2p$, since they traverse the inscribed $p$-gon twice.
It is shown in \cite[Theorem 2.4]{fong} that the minimizing index of these over-under curves is $2p$. The following proposition partially codifies this example:

\begin{prop}
\label{minindperpf}
For all closed geodesics $\gamma$ on $X_n$, $minind(\gamma) \geq per(\gamma)$.
\end{prop}

\begin{proof} 
This follows quickly from the fact that a geodesic on $X_n$ can not minimize on an open interval containing multiple edge points.
\end{proof}

While this lower bound is not always tight, it does provide a program (for small $p$) for demonstrating that $minind(X_p)=2p$. We already know that the over-under curves have minimizing index $2p$. By the proposition, we need only show that those closed geodesics with period less than $2p$ have minimizing index at least $2p$. 
Since for small $p$ there are relatively few geodesics with low period, one could hope to complete this classification program. Indeed, this program has been completed for the doubled triangle to show that $minind(X_3)=6$ (see \cite[Proposition 3.3]{fong} for details). However, as $p$ grows, the class of closed geodesics with period less than $2p$ grows, and this program becomes untenable. Recall that Theorem~\ref{infinity} only gives a limiting lower bound on the minimizing index of $X_p$, and that $minind(X_p)=2p$ is only conjecturally true.

Additionally, this proposition shows that if we have a sequence of geodesics with period tending to infinity, then their minimizing index tends to infinity as well. This argument will be used in the proof of Theorem~\ref{infinity}.

\section{Minimizing Properties of Closed Geodesics on the Doubled Disk}\label{disk}
Recall that the double disk $X_\infty$ consists of a ``front" and a ``back" unit disk, identified along their shared boundary circles. In this section, we will first classify all closed geodesics on the double disk (Proposition \ref{polygons and stars}) and then show that minimizing index equals period (Theorem~\ref{double disk minind}).

These results have two useful implications. First, there is a class of closed geodesics on double regular polygons which are identical in shape to those on the double disk (see Figures~\ref{Double disk figure 7} and \ref{Double disk figure 1}). Using our results on the double disk, we will be able to classify the minimizing indices of these geodesics on double polygons. This will give us a general upper bound on the minimizing index of $X_n$ (Corollary~\ref{divisor}). Second, we will use these results to study the relationship between minimizing index and Gromov-Hausdorff convergence, allowing us to prove Corollary~\ref{coro} in Section~\ref{limit}.

%Recall that $X_\infty$ is the Gromov-Hausdorff limit of $X_n$. For a closed geodesic $\gamma$ on the double disk, there may exist a sequence of closed geodesics $\gamma_n$ on double regular polygons where $\gamma_n$ converges to $\gamma$. This sequence would have the property that $minind(\gamma_n)$ grows without bound, but $minind(\gamma)$ is finite. In other words, the minimizing indices of these geodesics would commute poorly with Gromov-Hausdorff convergence.

\begin{figure}
\includegraphics[scale=.6]{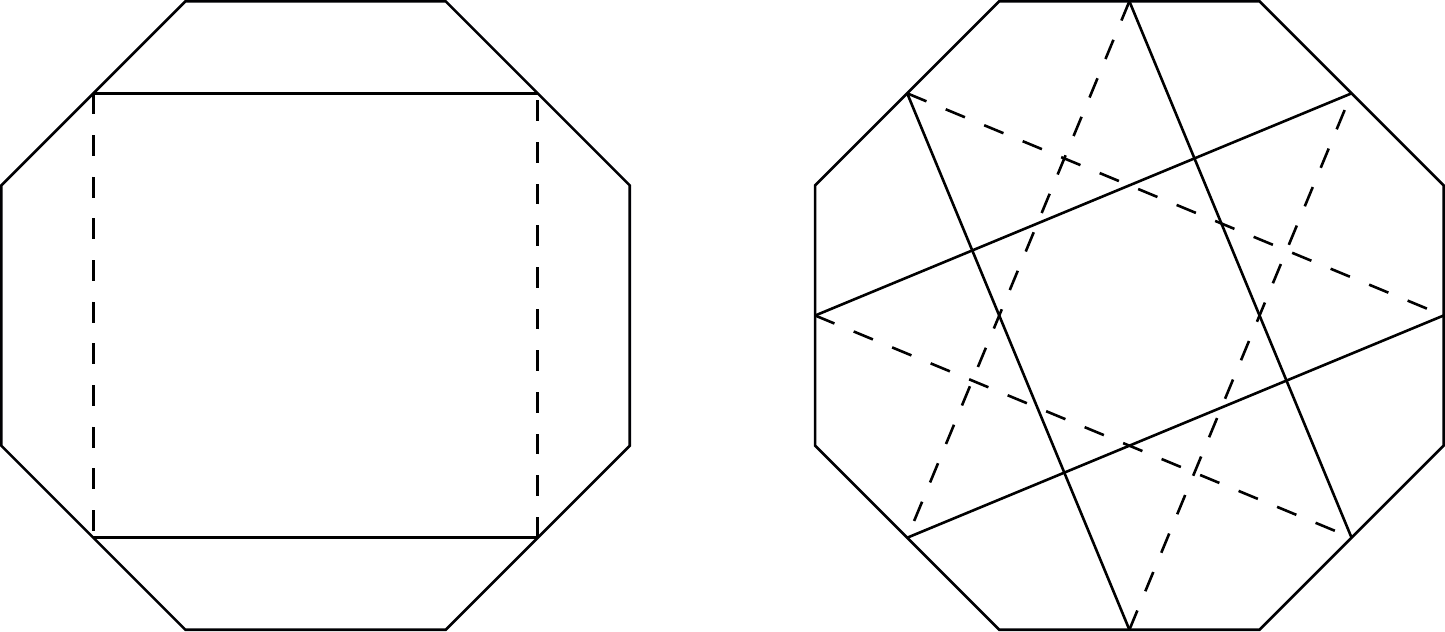}
\caption{A regular polygon geodesic and a regular star geodesic on $X_8$}
\label{Double disk figure 7}
\end{figure}

\begin{figure}
\includegraphics[scale=.6]{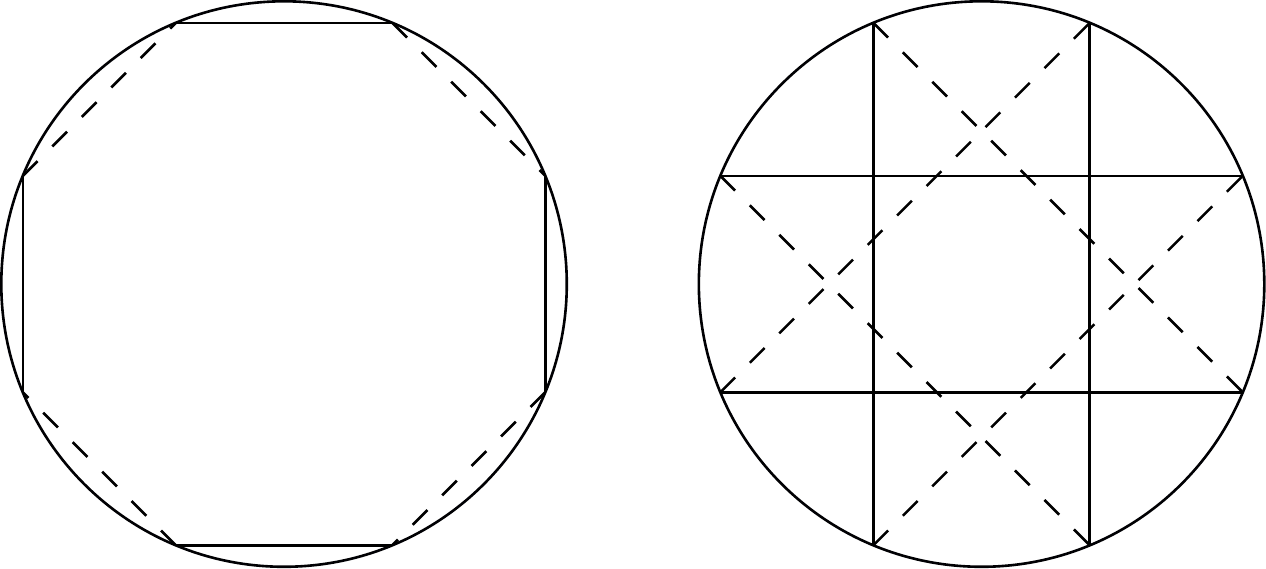}
\caption{A regular polygon geodesic and a regular star geodesic on $X_\infty$}
\label{Double disk figure 1}
\end{figure}

\subsection{Classification of double disk geodesics}\label{diskclassification}

We know that the only locally length-minimizing curve in Euclidean space is a straight line. Thus, every locally length-minimizing curve on $X_\infty$ is the concatenation of a series of line segments, which alternately lie on the front and back of $X_\infty$. Just like on the doubled polygons, the geodesics on $X_\infty$ have angle of incidence equal to angle of reflection (with the tangent line) whenever the geodesic intersects the boundary. We formalize this fact in the following lemma. See Figure~\ref{Double disk figure 4} for each of Lemma~\ref{angle bisector}, Proposition~\ref{polygons and stars}, and Theorem~\ref{double disk minind}. 

%To prove Proposition \ref{polygons and stars}, we must first prove a statement about the angles between consecutive segments.

\begin{lemma}
\label{angle bisector}
Let $\gamma = \overline{AQ} \cup \overline{QB}$ be a curve on $X_\infty$, where $A$ is on the front, $B$ is on the back, and $Q$ is on the boundary circle. Let $O$ be the center of the disk. Then $\gamma$ is locally length-minimizing if and only if $m\angle{AQO} = m\angle{BQO}$.
\end{lemma}

\begin{proof}
Again, this is just the statement that angle of incidence equals angle of reflection (with the tangent line) for locally length minimizing curves on the doubled disk. This fact follows as an application of Heron's solution to shortest path problem.
\end{proof}

\begin{prop}
\label{polygons and stars}
Let $\gamma$ be a closed curve on $X_\infty$. Then $\gamma$ is a geodesic if and only if $\gamma$ has the shape of a regular polygon or regular star inscribed in the unit circle, whose edges alternate between the front and back of $X_\infty$.
\end{prop}

By ``regular star," we mean a self-intersecting polygon with equal angles and equal sides (see Figure~\ref{Double disk figure 1}).
Here follows the proof of Proposition \ref{polygons and stars}.

\begin{proof}
Suppose $\gamma$ has the shape of a regular polygon or star inscribed in the unit circle, whose edges alternate between the front and back of $X_\infty$. We wish to prove that $\gamma$ is locally length-minimizing. We know the line segments are length-minimizing, so it will suffice to prove that $\gamma$ is locally length-minimizing at its vertices.

Let $\overline{PQ}$ and $\overline{QR}$ be two consecutive segments of $\gamma$, where $P$, $Q$, and $R$ lie on the unit circle. Let $O$ be the center of the unit disk. Since $\gamma$ has the shape of a regular polygon or star inscribed in the unit circle, we know $\overline{QO}$ bisects $\angle{PQR}$. In other words, $m\angle{PQO} = m\angle{RQO}$. Thus, by Lemma \ref{angle bisector}, $\gamma$ is locally minimizing at point $Q$.

We have shown that $\gamma$ is locally length-minimizing at all its vertices. Thus, $\gamma$ is locally length-minimizing. This means by definition that $\gamma$ is a geodesic.

Conversely, suppose $\gamma$ is a geodesic. Then $\gamma$ is the concatenation of a series of chords, which alternately lie on the front and back of $X_\infty$. Let $\overline{PQ}$ and $\overline{QR}$ be two consecutive segments of $\gamma$, where $P$, $Q$, and $R$ lie on the unit circle. Let $O$ be the center of the unit disk. Then by Lemma \ref{angle bisector}, $m\angle{PQO} = m\angle{RQO}$.

Let $\theta = m\angle{PQO} = m\angle{RQO}$. Then
\[PQ = QR = 2\cos{\theta}\]

Thus, all the line segments of $\gamma$ are equal in length. Since $\gamma$ is a closed curve consisting of equal chords of the unit circle reflected at equal angles, $\gamma$ must have the shape of a regular polygon or star inscribed in the unit circle.
\end{proof}

Note that the geodesic $\gamma$ can traverse a regular polygon or star more than once. If the polygon or star has an odd number of edges, then $\gamma$ must traverse it an even number of times, since the period of $\gamma$ must be an even number. We have now classified all closed geodesics on $X_\infty$. Next, we will determine the minimizing indices of these geodesics.

\subsection{Minimizing indices of double disk geodesics}

We are now ready to prove Theorem~\ref{double disk minind}. We first sketch the idea, then provide the details. Given a closed geodesic $\gamma$ on the double disk, we wish to find its minimizing index. We will pick an arbitrary subcurve of $\gamma$ with some fixed length, and prove that this subcurve is length-minimizing. To do so, we will attempt to construct the shortest possible path from one endpoint of this subcurve to the other in $X_\infty$. We will prove that this path must intersect the unit circle at one of four possible points. For each of these points, we will calculate the length of its correspoinding path, and finally prove that the shortest distance is achieved by the geodesic $\gamma$ itself. Here follows the proof of Theorem~\ref{double disk minind}.

\begin{proof}[Proof of Theorem~\ref{double disk minind}]
Let $\gamma$ be a closed geodesic on $X_\infty$. Let $m = per(\gamma)$. We must prove that $minind(\gamma) = m$. We know by Proposition \ref{minindperpf} that $minind(\gamma) \geq m$; thus, it will suffice to prove that $minind(\gamma) \leq m$.

Let $\overline{PQ}$ and $\overline{QR}$ be two consecutive segments of $\gamma$, where $P$, $Q$, and $R$ lie on the unit circle. Let $O$ be the center of the unit disk.

Lemma \ref{angle bisector} implies that $m\angle{PQO} = m\angle{RQO}$. Let $\theta = m\angle{PQO} = m\angle{RQO}$. Then $PQ = 2\cos{\theta}$. We know that $\gamma$ is a regular polygon or star (Proposition \ref{polygons and stars}), so every segment of $\gamma$ must have length $2\cos{\theta}$. It follows that the total length of $\gamma$ is $2m\cos{\theta}$.

In order to prove that $minind(\gamma) \leq m$, we must prove that every subcurve of $\gamma$ with length $2\cos{\theta}$ is length-minimizing. That is, we must prove that if some subcurve of $\gamma$ has length $2\cos{\theta}$ and endpoints $X$ and $Y$, then there is no path from $X$ to $Y$ on $X_\infty$ with length less than $2\cos{\theta}$.

\begin{figure}
\includegraphics[scale=.6]{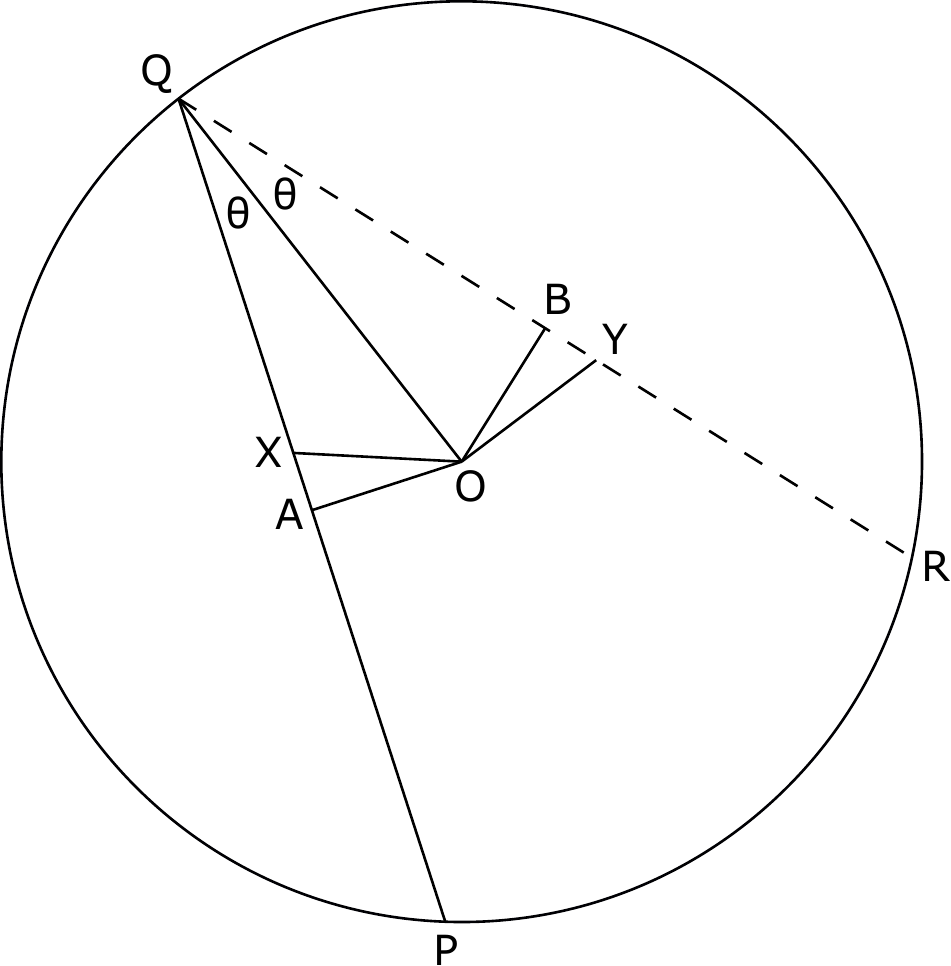}
\caption{Figure used in the proof of Lemma~\ref{angle bisector}, Proposition~\ref{polygons and stars}, and Theorem~\ref{double disk minind}}
\label{Double disk figure 4}
\end{figure}

Suppose $X$ and $Y$ are the endpoints of some subcurve of $\gamma$ with length $2\cos{\theta}$ (Figure \ref{Double disk figure 4}). Without loss of generality, suppose $X$ is on $\overline{PQ}$ and $Y$ is on $\overline{QR}$, so $XQ + QY = 2\cos{\theta}$. Let $Z$ be the point on the unit circle which minimizes the length $XZ + ZY$. If we can show that $XZ + ZY \geq 2\cos{\theta}$, this will prove that $\overline{XQ} \cup \overline{QY}$ is length-minimizing.

There are two degenerate cases: the case that $X = Y = O$, and the case that $X = Q$ or $Y = Q$. If $X = Y = O$, then $XZ + ZY = 2$ for all $Z$ on the unit circle, so the subcurve is length-minimizing. If $X = Q$ or $Y = Q$, then the subcurve is a line segment, and therefore length-minimizing. Thus, we can assume that $X, Y \neq O$ and $X, Y \neq Q$.

Let $A$ be the midpoint of $\overline{PQ}$, and let $B$ be the midpoint of $\overline{QR}$. Then $AQ + QB = XQ + QY = 2\cos{\theta}$, so
\[AX = BY\]

Furthermore,
\begin{gather*}
m\angle{XAO} = m\angle{YBO} = \frac{\pi}{2}\\
AO = BO = \sin{\theta}\
\end{gather*}

Thus, by the Pythagorean Theorem, $XO = YO$. Let $k = XO = YO$.

\begin{figure}
\includegraphics[scale=.6]{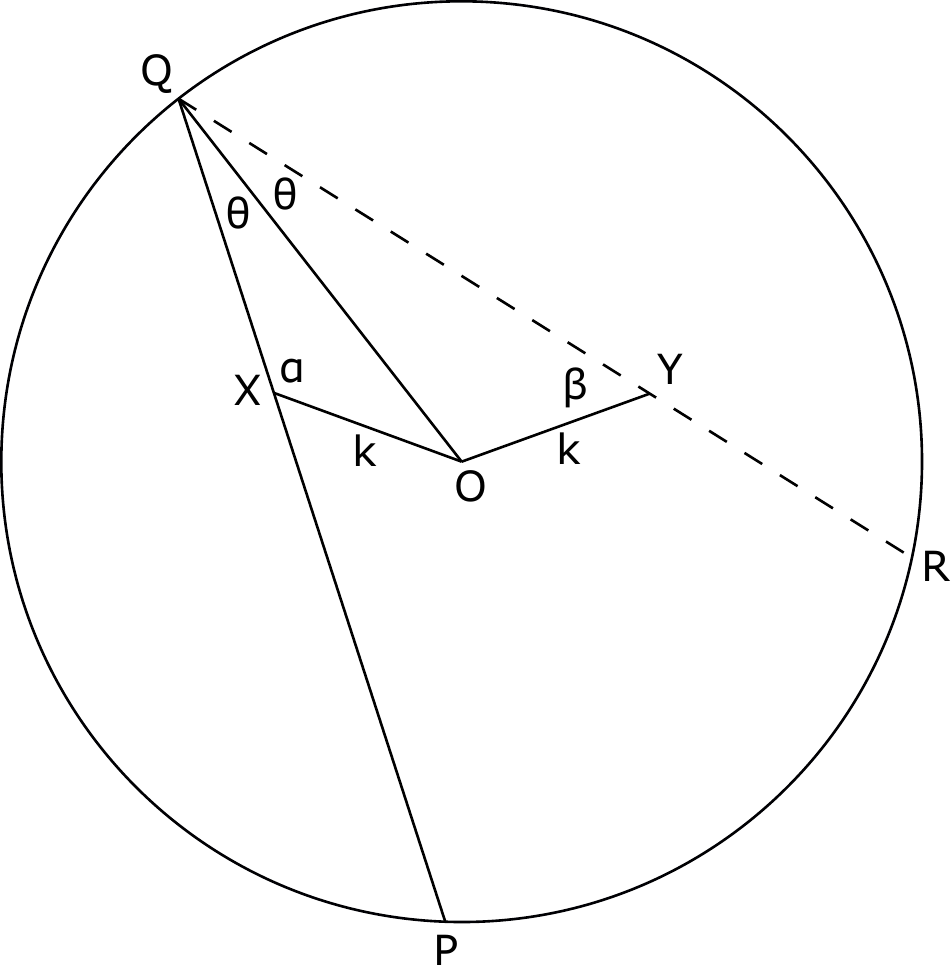}
\caption{Figure used in the proof of Theorem \ref{double disk minind}}
\label{Double disk figure 5}
\end{figure}

Let $\alpha = m\angle{QXO}$ and $\beta = m\angle{QYO}$ (Figure \ref{Double disk figure 5}). Using the law of sines on $\triangle{QXO}$, we get
\[\sin{\alpha} = \frac{\sin{\theta}}{k}\]

Using the law of sines on $\triangle{QYO}$, we get
\[\sin{\beta} = \frac{\sin{\theta}}{k}\]

Thus,
\[\sin{\alpha} = \sin{\beta}\]

It follows that either $\alpha = \beta$, or $\alpha + \beta = \pi$.

If $\alpha$ = $\beta$, then by the angle-angle-side property, $\triangle{XQO} \cong \triangle{YQO}$. This implies that $XQ = YQ = \cos{\theta}$, and therefore that $\alpha = \beta = \frac{\pi}{2}$. Thus, no matter what, $\alpha + \beta = \pi$.

The angles of $\square{XQYO}$ must sum to $2\pi$. Thus,
\begin{align*}
m\angle{XOY} &= 2\pi - \alpha - \beta - 2\theta\\
\Rightarrow m\angle{XOY} &= \pi - 2\theta
\end{align*}

Recall that $Z$ is the point on the unit circle which minimizes the length $XZ + ZY$. If $\overline{XZ} \cup \overline{ZY}$ is length-minimizing, then $\overline{XZ} \cup \overline{ZY}$ must be locally length-minimizing. By Lemma \ref{angle bisector}, it follows that $m\angle{XZO} = m\angle{YZO}$. Let $\varphi = m\angle{XZO} = m\angle{YZO}$.

Let $\zeta = m\angle{ZXO}$ and $\eta = m\angle{ZYO}$. Using the law of sines on $\triangle{ZXO}$, we get
\[\sin{\zeta} = \frac{\sin{\varphi}}{k}\]

Using the law of sines on $\triangle{ZYO}$, we get
\[\sin{\eta} = \frac{\sin{\varphi}}{k}\]

Thus,
\[\sin{\zeta} = \sin{\eta}\]

It follows that either $\zeta = \eta$ or $\zeta + \eta = \pi$.

\begin{figure}
\includegraphics[scale=.7]{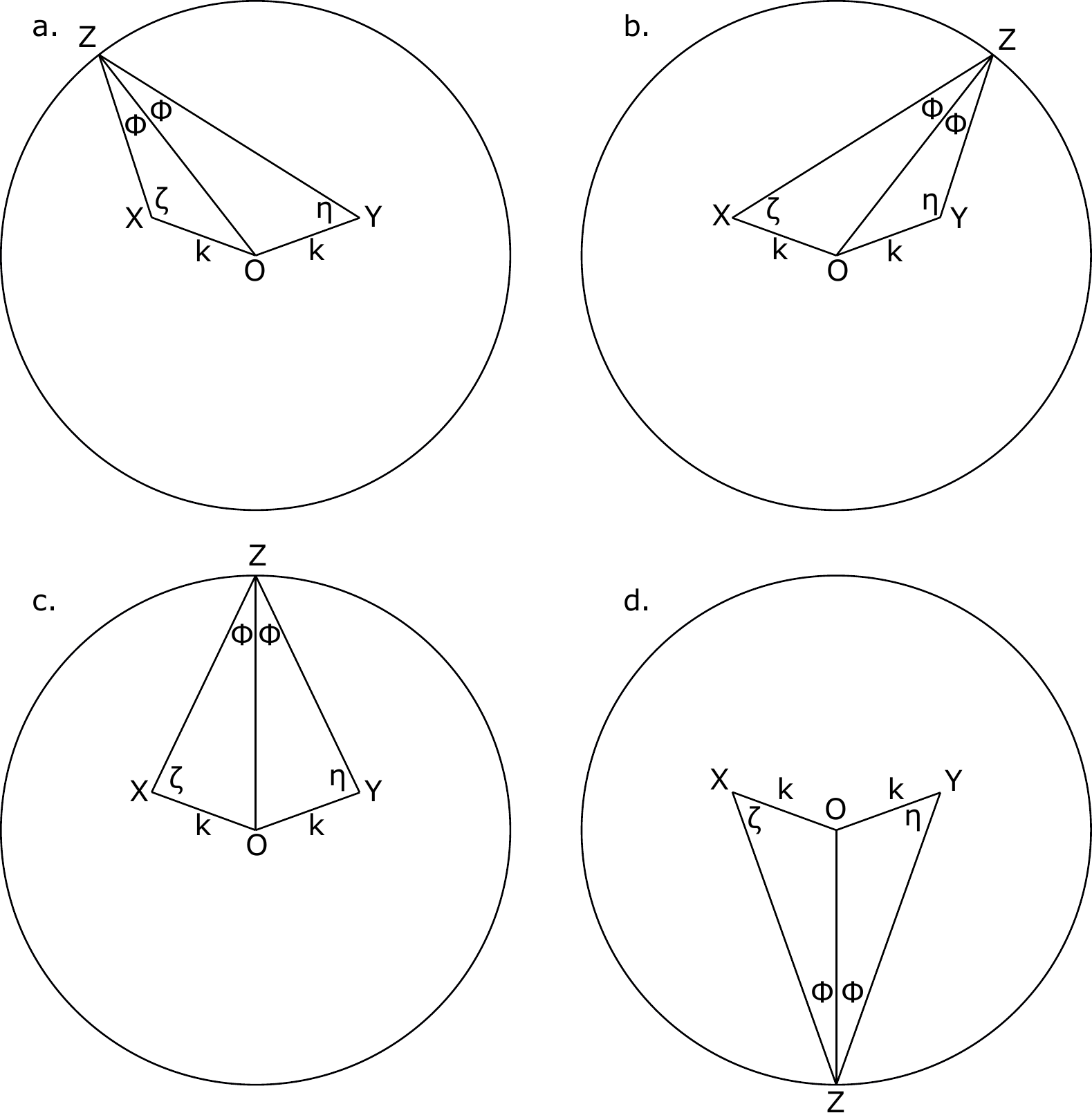}
\caption{Figure used in the proof of Theorem \ref{double disk minind}}
\label{Double disk figure 6}
\end{figure}

\textbf{Case 1: $\zeta + \eta = \pi$}

The angles of $\square{XZYO}$ must sum to $2\pi$. Thus,
\begin{align*}
m\angle{XZY} &= 2\pi - \zeta - \eta - (\pi - 2\theta)\\
\Rightarrow m\angle{XZY} &= 2\theta\\
\Rightarrow \varphi &= \theta
\end{align*}

It follows that
\[\sin{\zeta} = \sin{\eta} = \frac{\sin{\theta}}{k}\]

Thus, either
\[\zeta = \alpha, \eta = \beta \Rightarrow \square{XZYO} \cong \square{XQYO}\]

(Figure \ref{Double disk figure 6}a) or
\[\zeta = \beta, \eta = \alpha \Rightarrow \square{XZYO} \cong \square{YQXO}\]

(Figure \ref{Double disk figure 6}b). In either case,
\begin{align*}
XZ + ZY &= XQ + QY\\
\Rightarrow XZ + ZY &= 2\cos{\theta}
\end{align*}

\textbf{Case 2: $\zeta = \eta$}

By the angle-angle-side property, $\triangle{XZO} \cong \triangle{YZO}$. Thus, $m\angle{XOZ} = m\angle{YOZ}$. Recall that $m\angle{XOY} = \pi - 2\theta$. It follows that either
\[m\angle{XOZ} = m\angle{YOZ} = \frac{\pi}{2} - \theta\]

(Figure \ref{Double disk figure 6}c) or
\[m\angle{XOZ} = m\angle{YOZ} = \frac{\pi}{2} + \theta\]

(Figure \ref{Double disk figure 6}d). If $m\angle{XOZ} = m\angle{YOZ} = \frac{\pi}{2} + \theta$, then using the law of cosines on $\triangle{XZO}$, we find
\begin{align*}
(XZ)^2 &= 1 + k^2 - 2k\cos{\Big(\frac{\pi}{2} + \theta\Big)}\\
&= 1 + k^2 + 2k\sin{\theta}\\
&\geq 1\\
\Rightarrow XZ &\geq 1\\
\Rightarrow XZ + ZY &\geq 2\\
\Rightarrow XZ + ZY &\geq 2\cos{\theta}
\end{align*}

Alternatively, if $m\angle{XOZ} = m\angle{YOZ} = \frac{\pi}{2} - \theta$, then using the law of cosines on $\triangle{XZO}$, we find
\begin{align*}
(XZ)^2 &= 1 + k^2 - 2k\cos{\Big(\frac{\pi}{2} - \theta\Big)}\\
&= 1 + k^2 - 2k\sin{\theta}\\
\Rightarrow XZ &= \sqrt{1 + k^2 - 2k\sin{\theta}}\\
\Rightarrow XZ + ZY &= 2\sqrt{1 + k^2 - 2k\sin{\theta}}
\end{align*}

We wish to prove that $XZ + ZY \geq 2\cos{\theta}$. It will suffice to prove that $(XZ + ZY)^2 - (2\cos{\theta})^2 \geq 0$. We have
\begin{align*}
(XZ + ZY)^2 - (2\cos{\theta})^2 &= 4(1 + k^2 - 2k\sin{\theta}) - 4\cos^2{\theta}\\
&= 4(1 + k^2 - 2k\sin{\theta}) - 4(1 - \sin^2{\theta})\\
&= 4(k^2 - 2k\sin{\theta} + \sin^2{\theta})\\
&= 4(k - \sin{\theta})^2\\
&\geq 0
\end{align*}

Thus,
\[XZ + ZY \geq 2\cos{\theta}\]

We have proven that in all possible cases, $XZ + ZY \geq 2\cos{\theta}$. Since $Z$ is the point on the unit circle which minimizes the length $XZ + ZY$, it follows that $\overline{XQ} \cup \overline{QY}$ is length-minimizing.

We can conclude that every subcurve of $\gamma$ with length $2\cos{\theta}$ is length-minimizing. Since the total length of $\gamma$ is $2m\cos{\theta}$, it follows that $minind(\gamma) \leq m$, and therefore $minind(\gamma) = m$.
\end{proof}

We now have a complete classification of closed geodesics on the double disk (Proposition \ref{polygons and stars}), as well as their minimizing indices (Theorem \ref{double disk minind}). Next, we will use our results on the double disk to study the minimizing properties of double regular polygons.

\subsection{Upper bound on the minimizing index of $X_n$}

Recall from \cite[Proposition 2.1]{ade} that for $n$ even, $minind(X_{n})=2$. In this section we consider the case $n$ odd, and prove a general upper bound on the minimizing index of $X_n$ (Corollary~\ref{divisor}). We will then pose a conjecture about the exact minimizing index of $X_n$.

To prove Corollary \ref{divisor}, we will consider the class of closed geodesics on $X_n$ which only intersect the border of the regular $n$-gon at the midpoints of its sides. These geodesics are the shape of regular polygons and stars (Figure \ref{Double disk figure 7}) and are identical in shape to geodesics on the double disk $X_\infty$ (Figure~\ref{Double disk figure 1}). We will use this relationship to determine the minimizing indices of all such geodesics, thus providing the stated upper bound on the minimizing index of $X_n$.

\begin{lemma}
\label{inscribed polygons and stars}
Let $\gamma$ be a closed geodesic on $X_n$ which only intersects the border of the regular $n$-gon at midpoints. Then $minind(\gamma) = per(\gamma)$.
\end{lemma}

\begin{proof}
Let $m = per(\gamma)$. We will prove that $minind(\gamma) = m$. We know by Proposition \ref{minindperpf} that $minind(\gamma) \geq m$; thus, it will suffice to prove that $minind(\gamma) \leq m$.

Let $L$ be the total length of $\gamma$. In order to prove that $minind(\gamma) \leq m$, we must prove that any subcurve of $\gamma$ with length $\frac{L}{m}$ is length-minimizing.

We know that $\gamma$ has the shape of a regular polygon or star. Each segment of $\gamma$ has length $\frac{L}{m}$. Let $\overline{PQ}$ and $\overline{QR}$ be two consecutive segments of $\gamma$, and let $X$ and $Y$ be points on $\overline{PQ}$ and $\overline{QR}$, respectively, such that $XQ + QY = \frac{L}{m}$. We will prove that the subcurve $\overline{XQ} \cup \overline{QY}$ is length-minimizing.

Inscribe a circle $o$ in $X_n$ (Figure \ref{Double disk figure 8}). We know that $\gamma$ is a geodesic on $X_\infty$ (Proposition \ref{polygons and stars}). Moreover, we know that $\gamma$ is a $\frac{1}{m}$-geodesic on $X_\infty$ (Theorem \ref{double disk minind}). Thus, for all points $W$ on $o$, $XW + WY \geq XQ + QY$.

Let $Z$ be a point on the border of the regular $n$-gon. Then $Z$ is either on or outside of circle $o$. If $Z$ is on $o$, then
\[XZ + ZY \geq XQ + QY\]

If $Z$ is outside of circle $o$, then let $W$ be the intersection point of $\overline{XZ}$ and $o$. We know $XW + WY \geq XQ + QY$. Furthermore, by the triangle inequality,
\begin{align*}
WZ + ZY &\geq WY\\
\Rightarrow XZ + ZY &\geq XW + WY\\
\Rightarrow XZ + ZY &\geq XQ + QY
\end{align*}

Thus, for all points $Z$ on the border of the regular $n$-gon, $XZ + ZY \geq XQ + QY$. It follows that $\overline{XQ} \cup \overline{QY}$ is length-minimizing.

\begin{figure}
\includegraphics[scale=.6]{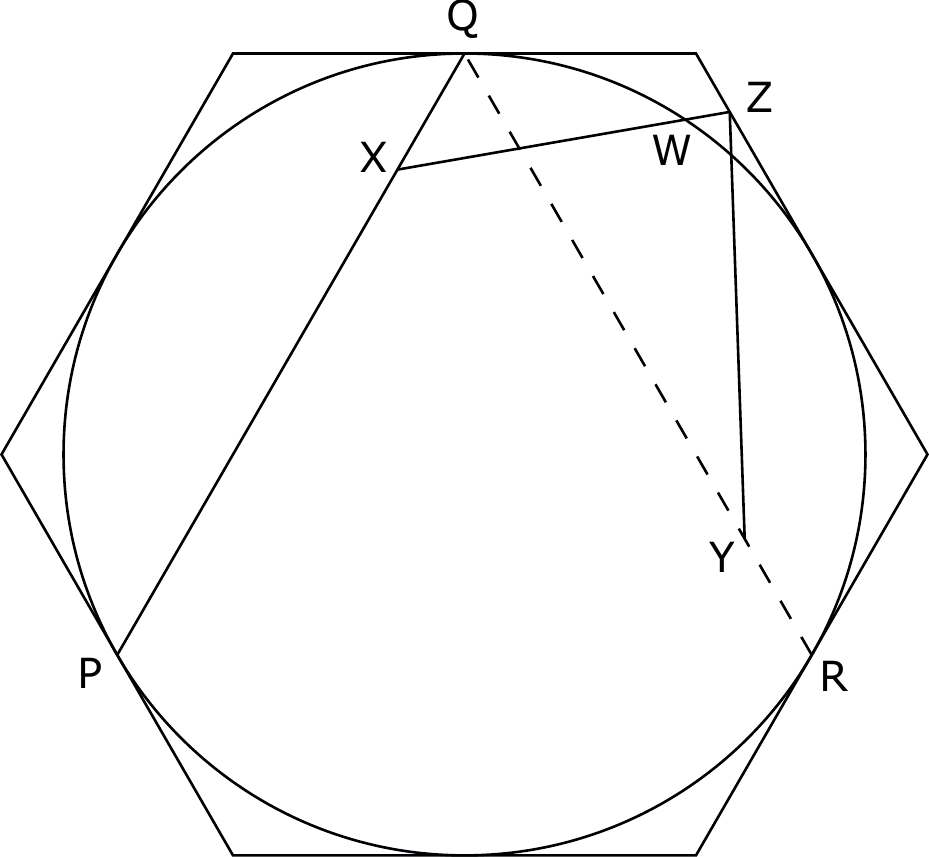}
\caption{Figure used in the proof of Lemma \ref{inscribed polygons and stars}}
\label{Double disk figure 8}
\end{figure}

We have shown that any subcurve of $\gamma$ with length $\frac{L}{m}$ is length-minimizing. Thus, $minind(\gamma) \leq m$. It follows that $minind(\gamma) = m$.
\end{proof}

We now give a proof of Corollary~\ref{divisor}.

\begin{proof}
Suppose $n$ is odd and $n \geq 3$. Let $d$ be the least prime divisor of $n$. Then there exists a closed geodesic $\gamma$ on $X_n$ such that $\gamma$ only intersects the border of the regular $n$-gon at midpoints, and $per(\gamma) = 2d$. The shape of $\gamma$ is a regular $d$-gon, which is traversed twice (Figure \ref{Double disk figure 9}). By Lemma \ref{inscribed polygons and stars}, we know $minind(\gamma) = 2d$. Thus, $minind(X_n) \leq 2d$.
\end{proof}

\begin{figure}
\includegraphics[scale=.9]{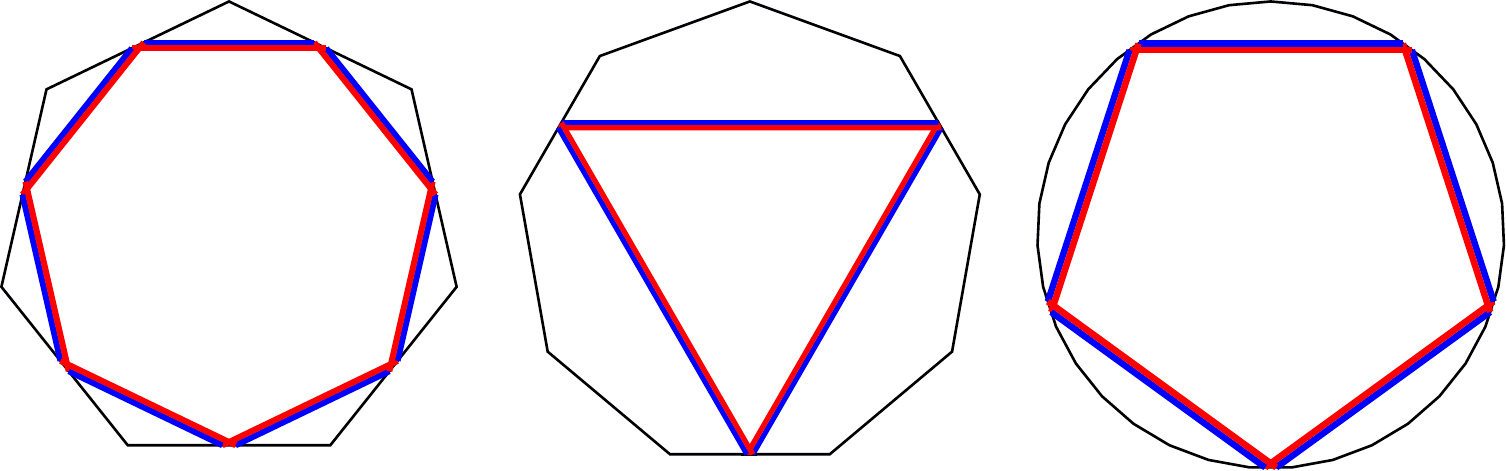}
\caption{From left ro right: a period 14 geodesic on $X_7$, a period 6 geodesic on $X_9$, and a period 10 geodesic on $X_{35}$}
\label{Double disk figure 9}
\end{figure}

%BEGIN COMMENT
\begin{comment}
Our upper bound on the minimizing index of $X_n$ motivates the following question:~for which odd $n$ is the bound of $2d$ achieved? It is shown in \cite[Proposition 3.3]{fong} that $minind(X_3)=6$. No odd number $n$ has yet been found for which $minind(X_n) < 2d$. However, finding a nonconstant lower bound on $minind(X_n)$ remains an open problem.

\begin{conjecture}
\label{2d conjecture}
For $n$ odd, $minind(X_n) = 2d$ where $d$ is the least prime divisor of $n$.
\end{conjecture}
\end{comment}
%END COMMENT

\section{Limit Properties of Minimizing Index for Doubled Polygons}\label{limit}

%If we fix the edge length of all $X_{p_i}$, there is a sequence of intersection points of $\gamma_i$ with the polygonal boundary of $X_{p_i}$ where the distance of an intersection point to the closest vertex of $X_{p_i}$ tends to zero. As geodesics do not minimize through vertex points, this ``proximity to a vertex" implies that $minind(\gamma_i)$ grows without bound.

This section is devoted to providing a rather long and detailed proof of Theorem~\ref{infinity}. To aid the reader, we first offer a brief summary of the proof. We introduce the notion of \emph{skip number} (Definition \ref{DefSkipNumber}), the number of vertices each segment of the geodesic passes in the counterclockwise direction. Using skip numbers, we define a notion of \emph{convergence} for a sequence of geodesics of equal period on increasing $X_n$ (Definition \ref{DefConvergence}). As $n$ grows without bound, $X_n$ tends to a double disk; in effect, a sequence of geodesics converges if their shape stabilizes with respect to the disk.

We then introduce the notion of \emph{vertex ratio} (Definition \ref{DefVertexRatio}), the ratio into which a geodesic splits a side of the polygon. We prove that in a convergent sequence of geodesics, the sequences of corresponding vertex ratios converge (Proposition \ref{lims}), and their limits obey certain arithmetic properties (Corollarys \ref{3lims} and \ref{arithmetic}). Using these properties, we show that at least one vertex ratio limit is 0 or 1 (Proposition \ref{Lim0}), which is sufficient to prove that the minimizing index grows without bound (Lemma \ref{0.5}). The intuition is as follows:~when a geodesic hits very close to a vertex, we can find a short segment of the geodesic through this intersection point which is not length-minimizing. The closer our geodesic gets to a vertex, the smaller this segment becomes, so the larger the geodesic's minimizing index must be.

%That is, it suffices to show $minind(\gamma_i) \rightarrow \infty$ for all such sequences. 

\subsection{Skip Numbers}

The notion of a sequence of \textit{skip numbers} is a key feature of closed geodesics on $X_n$ in defining convergence and the subsequent analysis. 

\begin{defi}
The $\textit{skip number}$ of a segment of a closed, oriented geodesic $\gamma$ on $X_n$ is the number of vertices of $X_n$ it passes in the counterclockwise direction (see Figure \ref{SkipNumber}).
\label{DefSkipNumber}
\end{defi}

\begin{figure}
\includegraphics[scale=0.4]{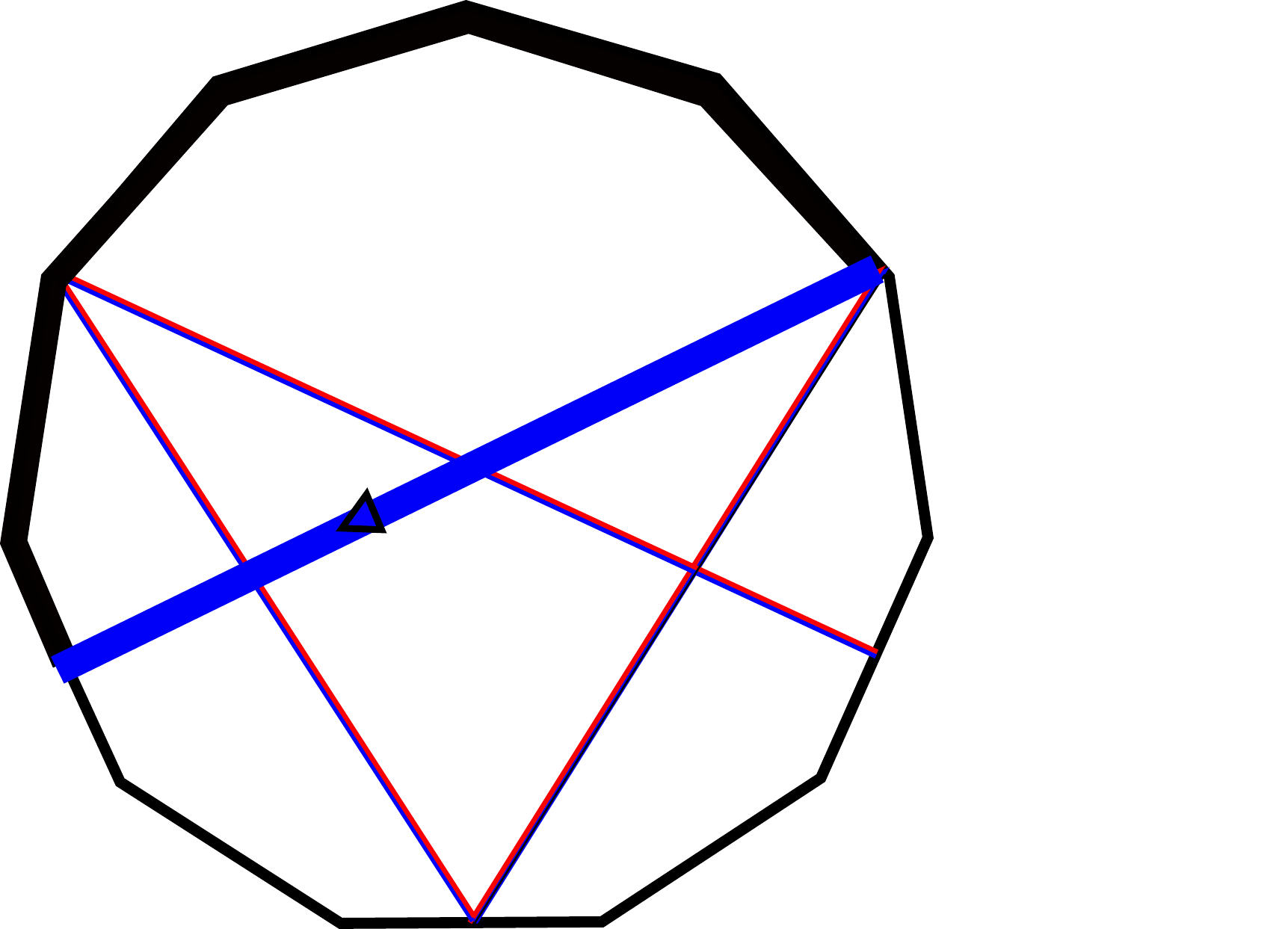}
\caption{A Skip Number of 5 of a Segment of a Closed Geodesic on $X_{11}$}
\label{SkipNumber}
\end{figure}
%A little nit-picky, but it's kind of hard to tell where the line segment hits relative to the vertex. Maybe some star/over-under curve would be a better illustration of skip number since there's no ambiguity there?
%I still think we need a clearer diagram here. If we're not going to switch out the geodesic at least label the 4 vertices to make it clear that the 4th one is actually skipped.

If we choose a starting point where a geodesic $\gamma$ intersects an edge of $X_n$ and traverse $\gamma$ in a given direction, there will be a sequence of skip numbers of length $per(\gamma)$ corresponding to the traversal. The following restriction on consecutive skip numbers in the sequence is instrumental in determining the shape of a closed geodesic on $X_n$.

\begin{prop} Consecutive skip numbers can differ by no more than $1$.
\label{skips}
\end{prop}
\begin{proof}

Let $\gamma$ be a geodesic on $X_n$, and let $s_1$ and $s_2$ be two consecutive skip numbers of $\gamma$. Without loss of generality, suppose $s_1 \geq s_2$. We will prove that $s_1 - s_2 \leq 1$.\\

\begin{figure}
\includegraphics[scale=0.4]{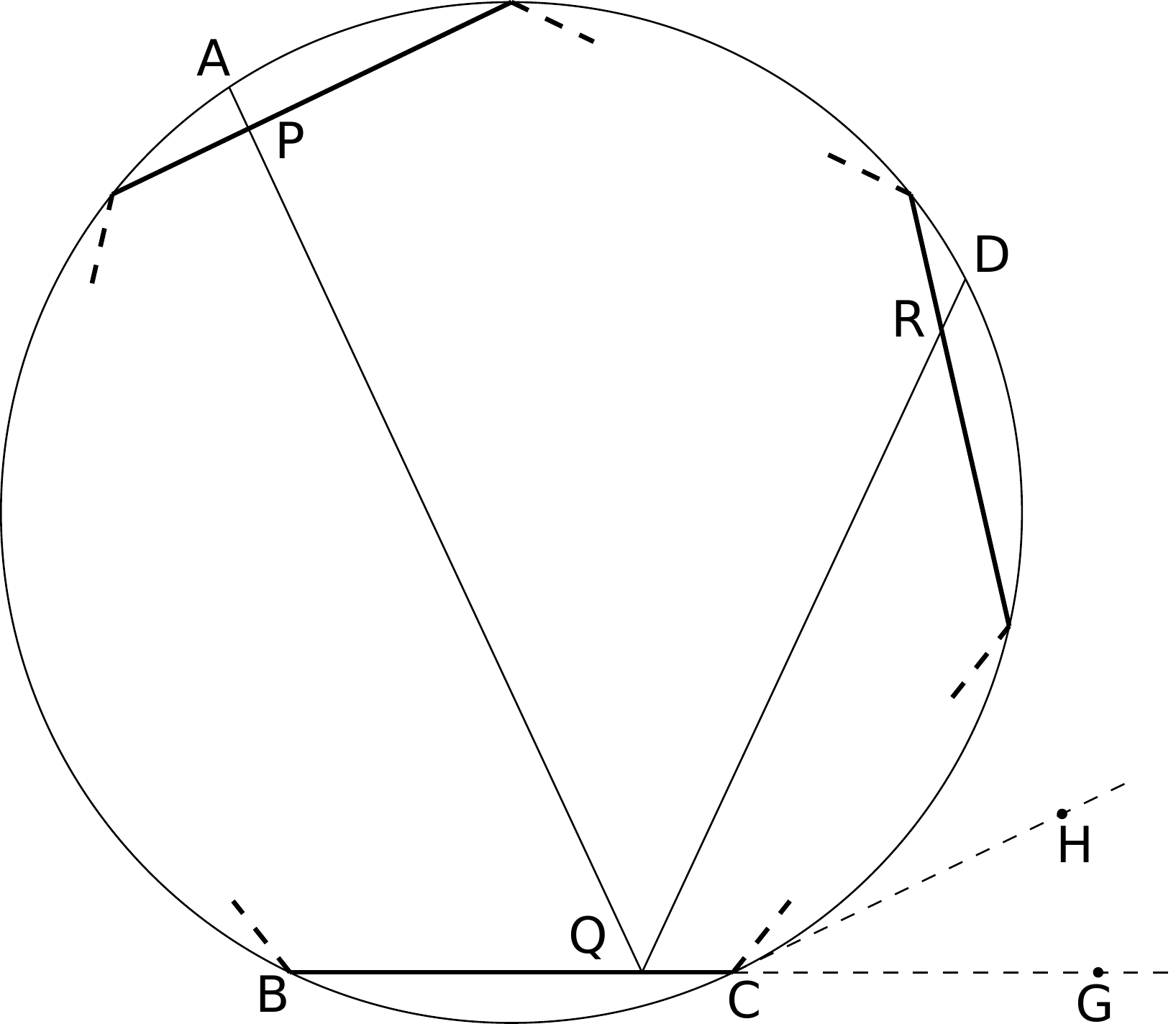}
\caption{Two adjacent segments of a geodesic on $X_n$}
\label{diff1pf}
\end{figure}

Let $\overline{PQ}$ and $\overline{QR}$ be the segments of $\gamma$ corresponding to skip numbers $s_1$ and $s_2$, respectively, as shown in Figure $\ref{diff1pf}$. Let $o$ be the circumscribed circle of $X_n$. Extend $\overrightarrow{QP}$ so that it intersects $o$ at point $A$, and extend $\overrightarrow{QR}$ so that it intersects $o$ at point $D$.\\

Let $\overline{BC}$ be the side of $X_n$ which contains point $Q$. Let $\theta = \angle{AQB} = \angle{DQC}$.\\

We have
\begin{align*}
s_1 &= \Bigl \lceil \frac{\widehat{AB}}{\frac{2\pi}{n}} \Bigr \rceil\\
s_2 &= \Bigl \lceil \frac{\widehat{CD}}{\frac{2\pi}{n}} \Bigr \rceil
\end{align*}
and
\[\widehat{AB} = 2\angle{ACB} < 2\theta\]

Extend $\overleftrightarrow{BC}$, and let $G$ be a point on $\overleftrightarrow{BC}$ such that $C$ is strictly between $B$ and $G$. Let $l$ be the line tangent to $o$ at point $C$, and let $H \neq C$ be a point on $l$ such that $H$ is inside $\angle{DQG}$. Then $\angle{GCH} = \frac{\pi}{n}$.\\

Thus,
\begin{align*}
\widehat{CD} &= 2\angle{DCH}\\
&= 2\angle{DCG} - \frac{2\pi}{n}\\
&> 2\theta - \frac{2\pi}{n}
\end{align*}

It follows that
\begin{alignat*}{3}
 \widehat{AB} - \widehat{CD} &< \frac{2\pi}{n}\\
\frac{\widehat{AB}}{\frac{2\pi}{n}} - \frac{\widehat{CD}}{\frac{2\pi}{n}} &< 1\\
\Bigl \lceil \frac{\widehat{AB}}{\frac{2\pi}{n}} \Bigr \rceil - \Bigl \lceil \frac{\widehat{CD}}{\frac{2\pi}{n}} \Bigr \rceil &\leq 1\\
 s_1 - s_2 &\leq 1
\end{alignat*}
\end{proof}

\subsection{Convergent Sequences of Geodesics}

We now proceed to introduce the notion of a convergent sequence of geodesics:

\begin{defi}
A sequence $\gamma_i$ of closed geodesics on $X_{n_i}$ with fixed period $n$ and skip number sequence $\{s_{ij}\}_{j = 1}^n$ is said to \emph{converge} if and only if there exists a constant $c > 0$ such that $\frac{s_{ij}}{n_i} \rightarrow c$ for all $1 \leq j \leq n$.
\label{DefConvergence}
\end{defi}

The geometric intuition for this definition follows from the fact that the $X_{n_i}$ converge in the Gromov-Hausdorff sense to $X_\infty$. Because the ratio of every skip number to the number of edges of $X_{n_i}$ tends to a uniform nonzero constant, we have convergence towards uniform arc length between segments on $X_\infty$. By Section \ref{diskclassification}, this resulting limit curve is a closed geodesic on $X_\infty$.

\begin{figure}
\begin{subfigure}{0.6\textwidth}
\includegraphics[scale = 0.8]{3nGons.pdf}
\subcaption{Convergence towards a Triangular Geodesic on $X_\infty$}
\end{subfigure}
\begin{subfigure}{0.6\textwidth}
\includegraphics[scale = 0.55]{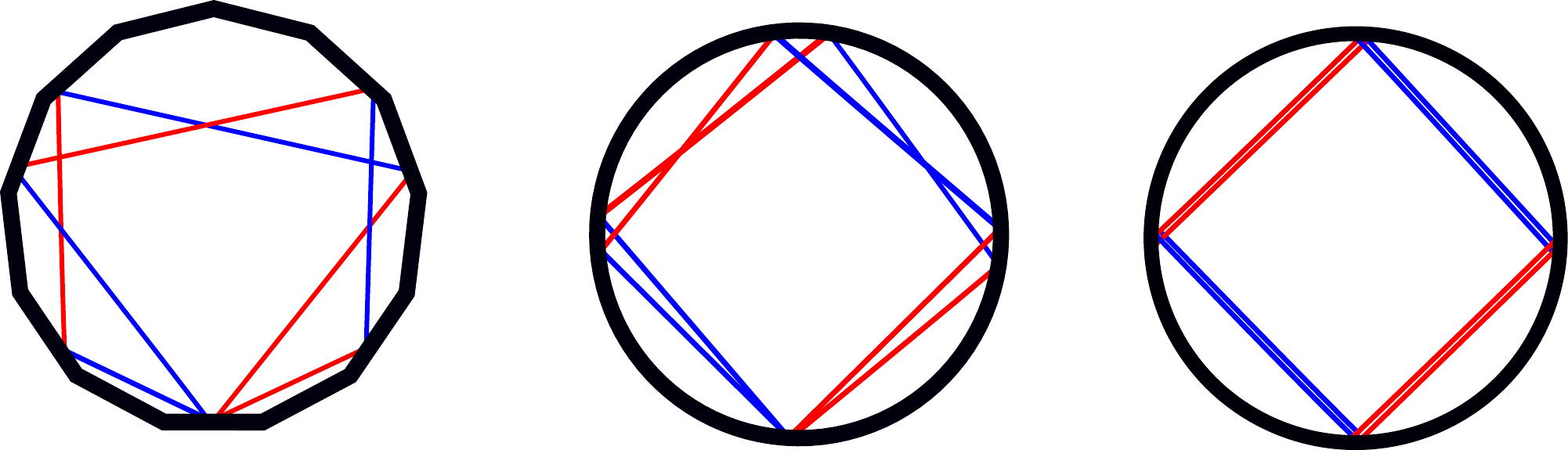}
\subcaption{Convergence towards a Square Geodesic on $X_\infty$}
\end{subfigure}
\caption{Examples of Convergent Sequences of Closed Geodesics}
\label{ConvergentGeodesics}
\end{figure}

Figure \ref{ConvergentGeodesics} shows examples of convergent sequences of geodesics. The following lemma now uses Proposition \ref{skips} along with the notion of geodesic convergence to show that we can restrict the proof of Theorem~\ref{infinity} to convergent sequences of geodesics.

\begin{lemma}
\label{convergent} 
A sequence $n_i \in \mathbb{N}$ satisfies $minind(X_{n_i}) \rightarrow \infty$ if and only if for all subsequences $n_{i_j}$ and all sequences of convergent geodesics $\gamma_i$ on $X_{n_{i_j}}$, we have $minind(\gamma_i) \rightarrow \infty$.
\end{lemma}

\begin{proof}
The forwards direction is immediate from the definition of minimizing index. We thus focus on the backwards direction. 

Assume for the sake of contradiction that $minind(X_{n_i}) \not\rightarrow \infty$. Then, we may extract a subsequence $n_{i_j}$ of $n_i$ such that $minind(X_{n_{i_j}}) \leq k$ for some fixed $k \in \mathbb{N}$. 
For ease of notation, let $n_i$ already be such a subsequence and let $\gamma_i$ be a closed geodesic on $X_{n_i}$ with $minind(\gamma_i) \leq k$. Since $per(\gamma_i) \leq minind(\gamma_i)$ (Proposition \ref{minindperpf}), there are a finite number of possibilities for $per(\gamma_i)$. 
Thus, we may refine $n_i$ by extracting a subsequence where $per(\gamma_i) = n \leq k$ for fixed $n$. Again we assume without loss of generality that $n_i$ is already this subsequence, and we let $\{s_{ij}\}_{j = 1}^n$ be a skip number sequence of $\gamma_i$. 
Because a geodesic segment must intersect two distinct segments of a doubled regular polygon, all skip numbers $s_{ij} < n_i$. Thus, $\Sigma_{j = 1}^n s_{ij} = l_i \cdot n_i$ with $l_i < n$.
Also, $l_i \in \mathbb{N}$ since $\gamma_i$ is closed and must thus wind around $X_{n_i}$ an integral number of times.

Therefore for each geodesic $\gamma_i$, there are only a finite number of possibilities for $l_i$ so we may again refine the sequence $n_i$ and assume without loss of generality that for all $\gamma_i$, $\Sigma_{j = 1}^n s_{ij} = l \cdot n_i$ for some fixed $l \in \mathbb{N}$ with $l < n$. 
Thus, for each geodesic $\gamma_i$, the average skip number is $\frac{l \cdot {n_i}}{n}$. By Proposition \ref{skips}, each skip number $s_{ij}$ satisfies $|s_{ij} - \frac{l \cdot n_i}{n}| < n$, so $|\frac{s_{ij}}{n_i} - \frac{l}{n}| < \frac{n}{n_i}$. Because $n$ is fixed, $n/n_i \rightarrow 0$. Thus because $l$ is fixed, $\displaystyle \lim_{i \rightarrow \infty} \frac{s_{ij}}{n_i} = \frac{l}{n}$ for all $j$.

Pulling back to a subsequence $n_{i_j}$ of $n_i$, we have constructed a convergent sequence of $1/k$-geodesics on $X_{n_{i_j}}$, reaching a contradiction. 
\end{proof}

For an example of this convergence, see Section~\ref{vshapesec} and Figure~\ref{vshape}.

%\begin{rema}
%For each geodesic $g_{n_{i_j}}$, we have a sequence of $d$ skip numbers and the difference $|s_m - l*{n_{i_j}}/d|$ has a finite number of possibilities. Thus, we may also refine the sequence so that the sequence of $d$ skip numbers are of the form $s_m = l*{n_{i_j}}/d + p_m$ for fixed $p_m$ across all $g_{n_{i_j}}$.
%\end{rema}

\subsection{Vertex Ratios and the Minimizing Index}

We now begin the proof of $minind(X_p) \rightarrow \infty$ by considering an arbitrary sequence of geodesics with even period $n$ on some sequence of prime-gons. Without loss of generality, we let $\gamma_p$ be a sequence of geodesics on $X_p$ for primes $p$, noting that we are really considering closed geodesics on $X_{p_i}$ for some subsequence $p_i$ of the primes.

For each $\gamma_p$, we have a sequence $\{s_{pj}\}_{j = 1}^n$ of $n$ skip numbers. As in the proof of Lemma \ref{convergent}, $|s_{pj} - \frac{\sum\limits_{j = 1}^n s_{pj}}{n}| < n$ 
%this expression just visually looks bad, is there a way to format this nicer?
by Proposition \ref{skips}. Because $s_{pj} \in \mathbb{N}$, there are a finite number of possible deviations each skip number of $\gamma_p$ can have from the average skip number of $\gamma_p$. 
Also as in the proof of Lemma \ref{convergent}, $\sum\limits_{j = 1}^n s_{pj} = l_p \cdot p$ for $l_p \in \mathbb{N}$ and $l_p < n$. Thus, there are also a finite number of possible values of $l_p$. 
Hence for sufficiently large $p$, we may partition the sequence $\gamma_p$ into subsequences where the $n$ skip numbers are of the form $s_{ij} = \frac{lp}{n} + d_j$, where the sum $\sum\limits_{j = 1}^n s_{pj} = lp$ for fixed $l < n$ and the deviation $d_j$ of a skip number from the average skip number is constant across all $\gamma_p$ in a subsequence. 

It now suffices to show that $minind(\gamma_p) \rightarrow \infty$ for such a subsequence. We let $d = \frac{n}{l}$ and $k_j = d \cdot d_j$ so that we may fix the skip number sequence of $\gamma_p$ as follows:

\[\frac{p + k_1}{d}, \frac{p + k_2}{d}, \frac{p + k_3}{d}, \ldots, \frac{p + k_n}{d}\]

Note that $\Sigma_{i = 1}^n k_i = 0$ because $\sum\limits_{j = 1}^n s_{pj} = lp = \frac{np}{d}$.
As an example, the V-shaped geodesics described in Section \ref{vshapesec} have $l = 2$, $n = 4$, $k_1=k_4 = -1$, and $k_2=k_3 = 1$.
%Having the V-shape example as a reference is nice! Not sure if it was in here before but it definitely helps readability.
%After moving the v-shaped discussion to Section 5, should we still have this example here? Previous comments indicate that this is a good example to include but now we're referencing things from later in the paper...

We now proceed to consider the proximity of these geodesics to vertices of $X_p$. We encapsulate this phenomenon with the notion of a vertex ratio. 

\begin{defi}
A segment of a closed geodesic on $X_p$ cuts an edge into a clockwise facing part $a$ and a counterclockwise facing part $b$. The $\textit{vertex ratio}$ is defined as $\frac{a}{a + b}$ (see Figure \ref{VertexRatio}).
\label{DefVertexRatio}
\end{defi}

\begin{figure}
\includegraphics[scale = 0.3]{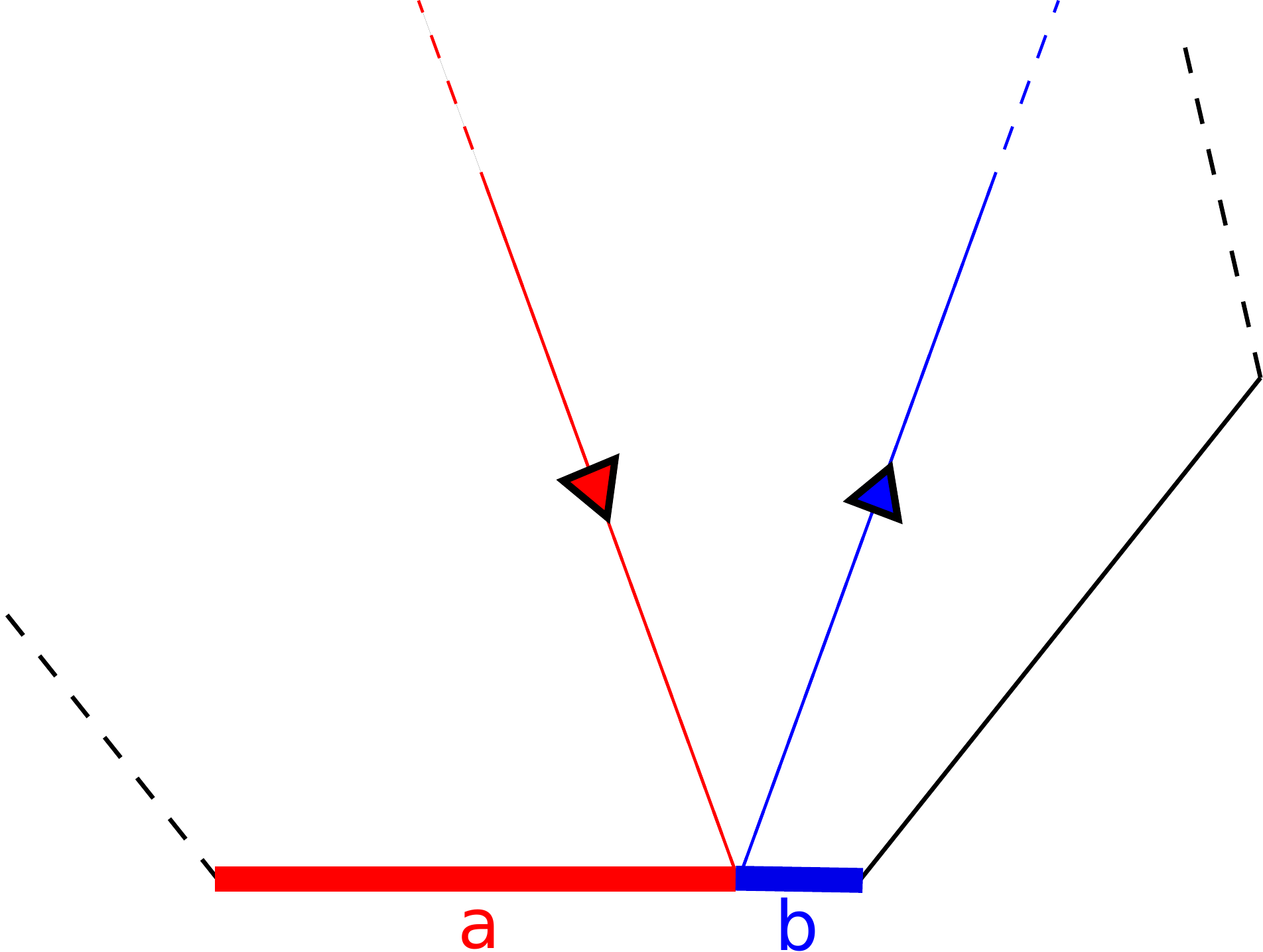}
\caption{The Vertex Ratio $\frac{a}{a + b}$ of the Incoming Geodesic Segment}
\label{VertexRatio}
\end{figure}

For a geodesic $\gamma_p$ as above, we have a sequence of $n$ vertex ratios. We denote the $i^{th}$ vertex ratio of $\gamma_p$ as $v_i(p)$ and $\lim_{p \rightarrow \infty} v_i(p)$ as $v_i^*$, if it exists. 

The next lemma provides the motivation for the remainder of the argument. After proving this lemma, all we need show is that for some $i$, the $i^{th}$ vertex ratio tends towards 0 or 1 so that for some sequence $v_p$ of vertices of $X_p$ (with fixed edge length), the distance of $\gamma_p$ from $v_p$  tends towards 0. 

\begin{lemma} 
\label{converge}
If $v_i^* = 0$ or  $v_i^* = 1$ for some $i$, then the minimizing index of the sequence of geodesics tends towards $\infty$. 
\end{lemma}

\begin{proof}
%The labels in the proof here do not match those in the diagram.  I get that the diagram is for the second part of the proof but this can only cause confusion.  Can we relabel to be consistent here?
Fix the inradius of our regular $p$-gons to be 1. Let the length of a side of the $p$-gon be $L$.

For any $i$, let $E_i$ be the edge of the $p$-gon that has vertex ratio $v_i(p)$, $Y$ be the point on $E_i$ that splits $E$ into this ratio, and $V$ be the vertex of the $p$-gon such that the distance $YV$ is $v_i(p)*L$.
Let $Y'$ be the point of intersection between the angle bisector of the angle at $V$ and the line passing through $Y$ perpendicular to $E$.  Let $x$ be the distance between $Y$ and $Y'$.

We claim that as $p \to\infty, x \to 0$.

Let $M$ be the midpoint of $E_i$ and $O$ be the center of the inscribed circle. Note that $\overline{OM}$ is an inradius of the $p$-gon, so the length of this segment is 1.  Furthermore, $\overline{OM}$ is perpendicular to $E_i$ and thus parallel to $\overline{YY'}$.
Thus $\triangle Y'YV \sim \triangle OMV$.

By similar triangles, we have $\frac{Y'Y}{OM} = \frac{YV}{MV}$, which simplifies $x = \frac{YV}{L/2}$ using the definition of $x$, $OM = 1$, and that $M$ is the midpoint of $E_i$.
But note that by definition, $\frac{Y'V}{L} = v_i(p)$.
This yields $x = 2v_i(p)$. Thus if ${v_i}^* = 0$, $x=0$. A symmetric argument shows that if ${v_i}^* = 1$, $x = 0$ as well.

Thus we have reduced our problem to showing that if $x = 0$, the minimizing index of the sequence of geodesics tends towards $\infty$.

Let $(p_i)_{i = 0}^\infty$ be an increasing sequence of prime numbers with $p_0 \geq 3$. Let $(\gamma_i)_{i = 0}^\infty$ be a sequence of closed geodesics where $\gamma_i$ is a geodesic on $X_{p_i}$ and such that $x = 0$. We will prove that
\[\lim_{i\to\infty} minind(\gamma_i) = \infty\]

Suppose that for all $p_i$, $\gamma_i$ has period $n$ and skip numbers $\frac{p_i + k_1}{d}, \frac{p_i + k_2}{d}, ..., \frac{p + k_n}{d}$, where $n$, $d$ and $k_1, k_2, ..., k_n$ are integers independent of $i$. We once again take $X_{p_i}$ to have inradius $1$ for all $p_i$.\\

For all $p_i$, let $\overline{P_iQ_i}$ and $\overline{Q_iR_i}$ be two consecutive segments of $\gamma_i$. Let $\overline{B_iC_i}$ be the side of $\chi_{p_i}$ which contains $Q_i$. Let $D_i$ be the other vertex of $\chi_{p_i}$ adjacent to $C_i$. Let $l_i$ be the bisector of $\angle{B_iC_iD_i}$. Let $m_i$ be the line though $Q_i$ perpendicular to $\overline{B_iC_i}$. Let $T_i$ be the intersection of $l_i$ and $m_i$. See Figure $\ref{anglebisectorpf}$ for a diagram.\\

\begin{figure}
\includegraphics[scale=0.5]{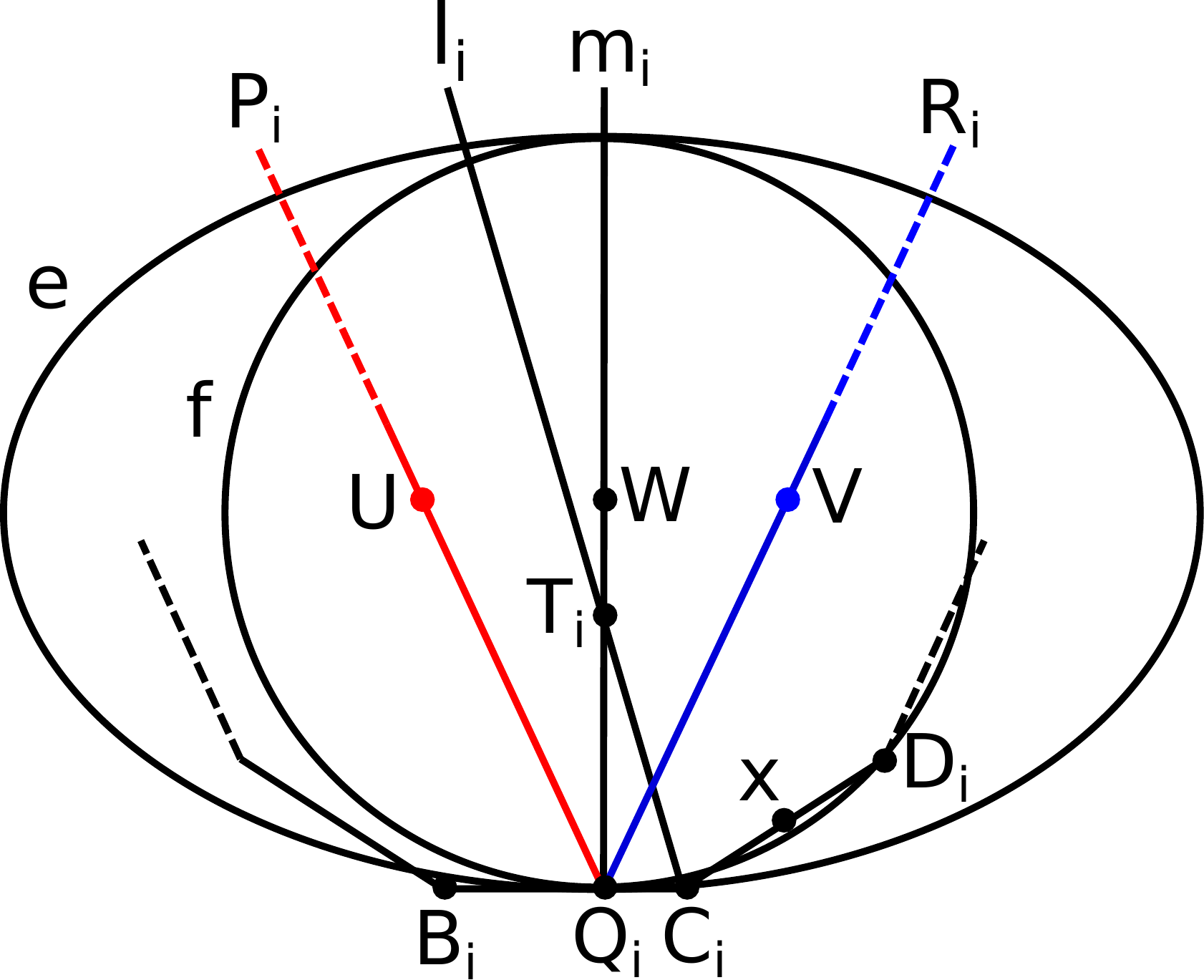}
\caption{Showing that if $x\to 0$, the minimizing index tends towards $\infty$}
\label{anglebisectorpf}
\end{figure}

Let $x_i = \lvert \overline{Q_iT_i} \rvert$. Then $\lim_{i\to\infty} x_i = x = 0$.\\

For all $p_i$, let $\theta_i = \angle{P_iQ_iB_i} = \angle{R_iQ_iC_i}$. Then
\[\lim_{i\to\infty} \theta_i = \frac{2\pi}{d}\]
Since all $\chi_{p_i}$ have inradius $1$, $\chi_{p_i}$ converges to the unit circle, and the length of each segment of $\gamma_i$ converges to a nonzero constant. Meanwhile, $x_i$ converges to $0$. Thus, there exists an integer $j$ sufficiently large that for all $i \geq j$,
\[\lvert \overline{P_iQ_i} \rvert, \lvert \overline{Q_iR_i} \rvert > x_i \csc(\theta_i)\]
Let $L_i$ be the length of $\gamma_i$. We will prove that for all $i \geq j$,
\[minind(\gamma_i) \geq \frac{L_i}{2 x_i \csc(\theta_i)}\]
Fix $i \geq j$. By way of contradiction, suppose $minind(\gamma_i) < \frac{L_i}{2 x_i \csc(\theta_i)}$. Then there exists $m > 2 x_i \csc(\theta_i)$ such that all subcurves of $\gamma_i$ with length less than or equal to $m$ are length-minimizing.\\

Let $U$ be a point on $\overline{P_iQ_i}$ and $V$ be a point on $\overline{Q_iR_i}$ such that $x_i \csc(\theta_i) < \lvert \overline{Q_iU} \rvert = \lvert \overline{Q_iV} \rvert \leq \frac{m}{2}$. Then $\widehat{UV} \subseteq \gamma_i$ must be length-minimizing.\\

Let $W$ be the midpoint of $U$ and $V$. Since $\lvert \overline{Q_iU} \rvert = \lvert \overline{Q_iV} \rvert$ and line $m_i$ is the bisector of $\angle{UQ_iV}$, we know $W$ lies on $m_i$. Since $\lvert \overline{Q_iU} \rvert > x_i \csc(\theta_i)$, we know $\lvert \overline{Q_iW} \rvert > x_i$, so $T_i$ is strictly between $Q_i$ and $W$. This means that $W$ and $D_i$ lie on the same side of line $l_i$. Thus, $W$ is closer to $\overleftrightarrow{C_iD_i}$ then it is to $\overleftrightarrow{B_iC_i}$. In other words, there exists a point $X$ on $\overline{C_iD_i}$ such that $\lvert \overline{WX} \rvert < \lvert \overline{WQ_i} \rvert$.\\

Let $e$ be the ellipse with foci at $U$ and $V$ that passes through $Q_i$, and let $f$ be the circle centered at $W$ that passes through $Q_i$. Then ellipse $e$ contains circle $f$. We know $X$ is in the interior of $f$, so $X$ must be in the interior of $e$. It follows that $\lvert \overline{UX} \rvert + \lvert \overline{XV} \rvert < \lvert \overline{UQ_I} \rvert + \lvert \overline{Q_IV} \rvert$. Thus, $\widehat{UV} \subseteq \gamma_i$ is not length-minimizing, a contradiction.\\

We have proven by contradiction that for all $i \geq j$,
\[minind(\gamma_i) \geq \frac{L_i}{2 x_i \csc(\theta_i)}\]
We know
\begin{align*}
L_i &> 1\\
\lim_{i\to\infty} x_i &= 0\\
\lim_{i\to\infty} \theta_i &= \frac{2\pi}{d}
\end{align*}
Thus,
\[\lim_{i\to\infty} minind(\gamma_i) = \infty\]

\end{proof}

\subsection{Convergence of Vertex Ratios}

We now proceed to focus on the convergence of vertex ratios, with the previous lemma in mind. The following formula relates consecutive vertex ratios, $v_i(p)$ and $v_{i+1}(p)$, in the vertex ratio sequence to the skip number $s$ of the segment joining them and the angle $\theta$ of inclination of that segment from the edge corresponding to $v_i(p)$:
%A diagram demonstrating what theta is might be useful. I do see one significantly later for the development/determining an angle, so maybe just refer to that?
\begin{form}
\label{vertexratio}
\[v_{i + 1}(p) = \csc \left(\frac{2s\pi}{p} - \theta \right) \left( \left(1 - v_i(p) \right)\sin(\theta) - \frac{\cos \left(\theta - \frac{\pi}{n} \right) - \cos \left(\frac{(2s - 1)\pi}{n} - \theta \right)} {2\sin \left(\frac{\pi}{n} \right)} \right)\]
\end{form}

\begin{proof}
Fix the edge length of $X_n$ to be 1, so that we may form a quadrilateral out of the geodesic segment, segments of length $1 - v_i(p)$ and $v_{i + 1}(p)$, and a chord of length $2r\sin(\frac{(s-1)\pi}{p})$, where $r = \frac{1}{2\sin(\pi/p)}$ is the circumradius of $X_p$ (see Figure \ref{quad}). 

\begin{figure}
\centering
\includegraphics[scale = 0.2]{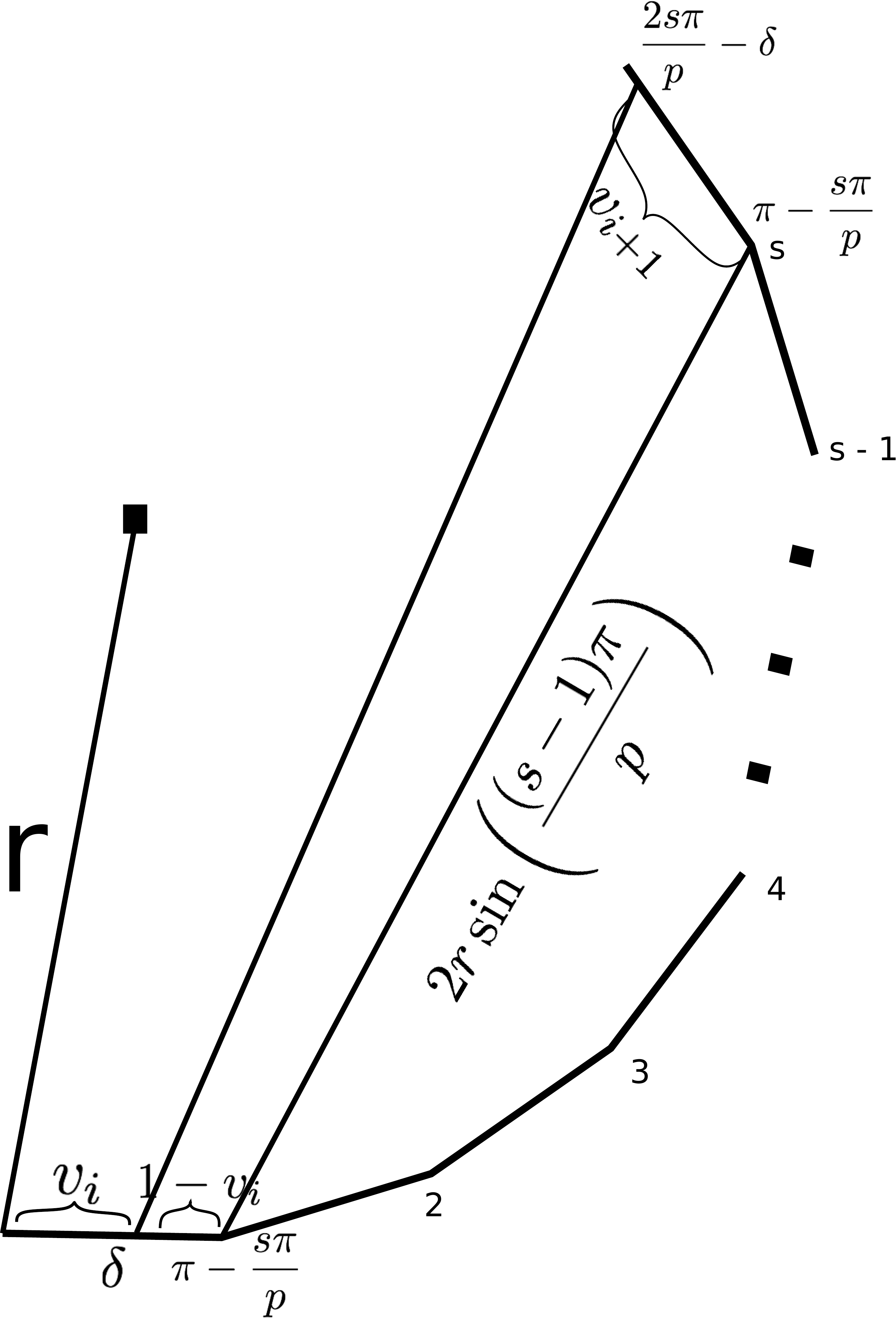}
\caption{Quadrilateral Formed from by the Geodesic Segment, Chord Between Edges, and Vertex Ratio Segments}
\label{quad}
\end{figure}

Using geometric arguments, we may deduce that the angles of the quadrilateral incident to the chord both have measure $\pi - \frac{s\pi}{p}$, while the remaining angle has measure $\frac{2s\pi}{p} -\theta$. In \cite{stack} it is given a formula and a proof for deriving the lengths of the sides of a quadrilateral given the angles and lengths of two sides. From this, we deduce

\[v_{i + 1}(p) = \csc \left(\frac{2s\pi}{p} - \theta \right) \left( \left(1 - v_i(p) \right)\sin(\theta) - \frac{2\sin \left(\frac{(s-1)\pi}{p} \right)}{2\sin \left(\frac{\pi}{p} \right)}\sin \left(\pi - \frac{s\pi}{p} + \theta \right) \right)\]

The desired formula then arises from the trigonometric identities $\sin(\pi - \alpha) = \sin(\alpha)$ and $2\sin(\alpha)\sin(\beta) = \cos(\alpha - \beta) - \cos(\alpha + \beta)$.
\end{proof}

In order to establish a root for the recursive formula above, we let $v_0(p) = 0.5$. This step will be justified later on in Lemma \ref{0.5}. We now proceed to show that $v_i(p)$ converges for each $i$.
%When we let v_0(p) = 0.5, we're doing that for all p, right? That should be made clear.
%Is "root" the right word to be using here? Something like "base" sounds better, but that may just be me.

First, we compute $\theta$ explicitly as a function of $p$ and the skip number sequence of $\gamma_i$.

The \textit{development} of a geodesic is the reflective shape formed by successively unfolding $X_p$ when we traverse the geodesic as a straight line. By examining the sum of the vectors joining the midpoints of the polygons in the development in a method similar to that of Fuchs \cite{fuchs}, it can be shown that for skip numbers $s_1, s_2, \ldots, s_n$, the angle of inclination from the initial edge is 
%Where does the -pi/2 come from? I understand where the arg(zeta_p stuff) comes from but not the subtraction.

\[\arg \left(\sum\limits_{i = 1}^{n/2} \zeta_p^{S_{2i - 1}} + \sum\limits_{i = 1}^{n/2} \zeta_p^{p/2 + S_{2i}} \right) - \frac{\pi}{2}\]

where $S_j = \sum\limits_{k = 1}^j (-1)^{k + 1} s_k$ (see Figure \ref{angdevelopment}). That is, $S_i = \mychi_{2\mathbb{Z} + 1}(i) \frac{p}{d} + \sum\limits_{j = 1}^i (-1)^{j + 1} \frac{k_j}{d}$, where $\mychi_{2\mathbb{Z} + 1}$ is the characteristic function of the odds. 

\begin{figure}
\centering
\includegraphics[scale = 0.24]{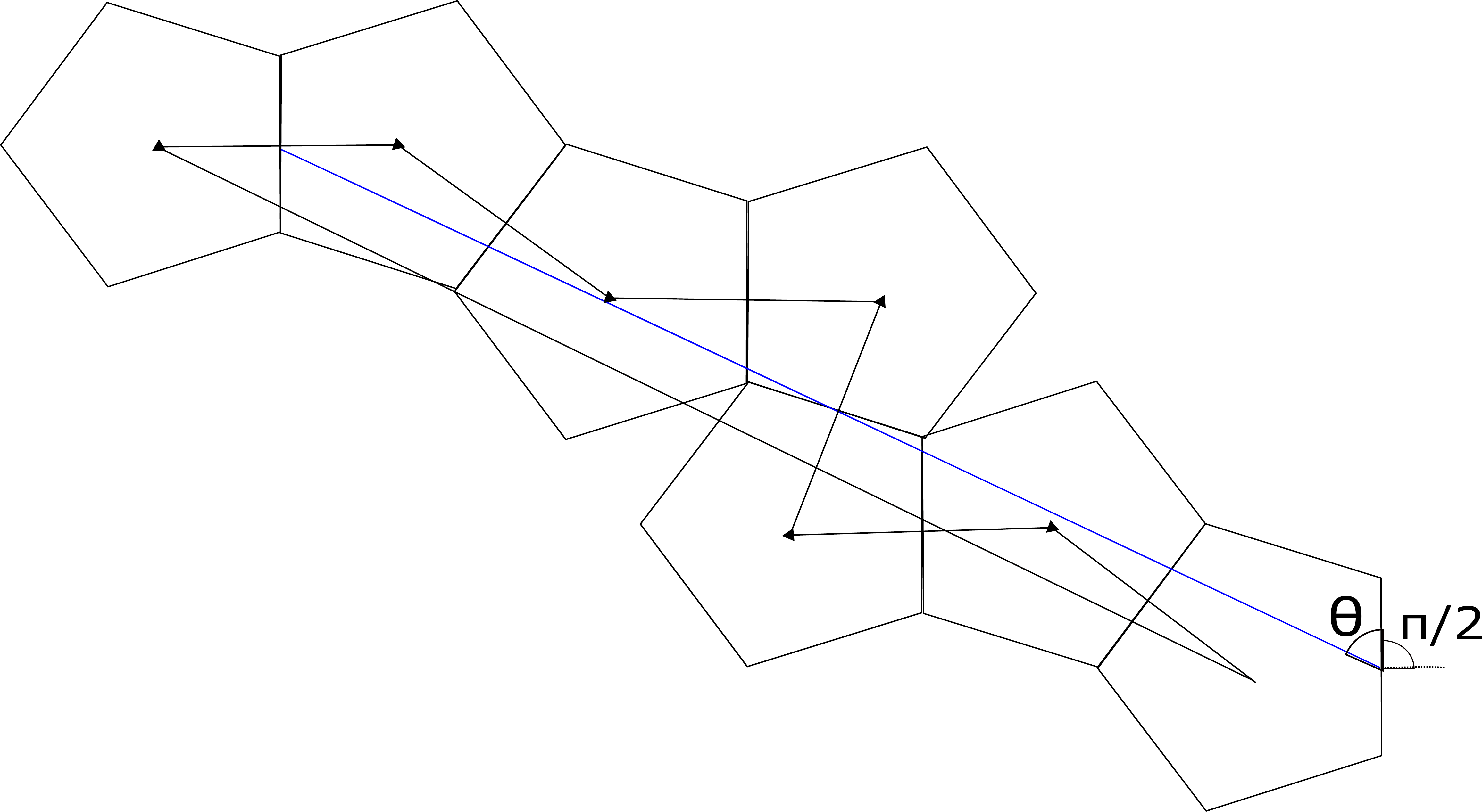}
\caption{Illustration of Computation of Angle of Inclination $\theta$ of Geodesic}
\label{angdevelopment}
\end{figure}
%Why is the first center not labeled/identified in this diagram while the rest are? Also what is the purpose of showing the pi/2 angle here -- is it related to the -pi/2 in the formula above?

The angle is equivalent to

\[\arctan \left(\frac{\sum\limits_{i = 1}^n (-1)^{i + 1} \sin \left( \frac{2\pi S_i}{p} \right) }{\sum\limits_{i = 1}^n (-1)^{i + 1} \cos \left( \frac{2\pi S_i}{p} \right) } \right) + \frac{\pi}{2}\]

%\[\arg \left(\zeta_p^{s_1} + \zeta_p^{p/2 + (s_1 - s_2)} + \zeta_p^{s_1 - s_2 + s_3} + \zeta_p^{p/2 + (s_1 - s_2 + s_3 - s_4)} + \ldots + \zeta_p^{p/2 + (s_1 - s_2 + s_3 - s_4 + \ldots - s_n)} \right)  -  \pi/2 = \]

%\[ = \arctan (\frac{\sin \left(\frac{2\pi s_1}{p} ]\right) - \sin(\frac{2\pi(s_1 - s_2)}{p}) + \sin(\frac{2\pi(s_1 - s_2 + s_3)}{p}) - \sin(\frac{2\pi(s_1 - s_2 + s_3 - s_4)}{p}) + \ldots}{\cos(\frac{2\pi s_1}{p}) - \cos(\frac{2\pi(s_1 - s_2)}{p}) + \cos(\frac{2\pi(s_1 - s_2 + s_3)}{p}) - \cos(\frac{2\pi(s_1 - s_2 + s_3 - s_4)}{p}) + \ldots} \]
%\[\frac{ - \sin(\frac{2\pi(s_1 - s_2 + s_3 - s_4 + \ldots - s_n)}{p}}{- \cos(\frac{2\pi(s_1 - s_2 + s_3 - s_4 + \ldots - s_n)}{p}} ) - \frac{\pi}{2}\]

Let us denote this expression for the skip number sequence $\frac{p + k_1}{d}, \frac{p + k_2}{d}, \frac{p + k_3}{d}, \ldots, \frac{p + k_n}{d}$ by $\theta(p)$. 

It follows that 

\[\lim_{p \to \infty} \theta(p) = \arctan \left(\frac{\frac{n}{2}\sin(\frac{2\pi}{d})}{\frac{n}{2}\cos(\frac{2\pi}{d}) - \frac{n}{2}} \right) + \frac{\pi}{2} = \arctan \left(\frac{\sin(\frac{2\pi}{d})}{\cos(\frac{2\pi}{d}) - 1} \right) + \frac{\pi}{2} =
\frac{\pi}{d}\] 

This follows not only algebraically but also geometrically from the fact that the ratio of any skip number of $\gamma_i$ to $p$ tends towards $\frac{1}{d}$, and so on the doubled disk, angle of inclinations of $\lim(\gamma_i)$ intercept arcs of length $2\pi/d$. 
%How exactly does it follow algebraically? I'd like to see the argument (or at least a sketch) written up.

We now proceed to compute $\lim_{p \to \infty} v_i(p)$.
%by induction on the limit of Formula \ref{vertexratio} as $p \rightarrow \infty$.

%Super nitpick: why not call this v_i* here, we've already defined it

\begin{prop}
\label{lims}
$v_i^* = \lim_{p \to \infty} v_i(p)$ exists for all $i$ and 

\[v_{i + 1}^* = 1 - v_i^* - \frac{k_{i + 1}}{d} + 2\sum\limits_{j = 1}^n \frac{(-1)^{j+1}(n - j + 1)(k_{j + i})}{dn}\]

where we let the $k_i$ sequence be periodic so that $k_i = k_{i + n}$.
\end{prop}

\begin{proof} 
The proof follows from taking the limit as $p$ tends to $\infty$ of Formula \ref{vertexratio} applied to the $v_i(p)$  sequence at hand. We will show the result for the base case $v_1^*$ given that $v_0^*$ exists ($v_0^* = 0.5$). Induction may then be used for $v_i^*$ using the same computations that will follow, only with the use of $v_{i - 1}(p)$ instead of $v_0(p)$ and $k_{j + (i - 1)}$ instead of $k_j$ for all $k_j$. $\theta(p)$ is also modified, but this will not change the computations.% The full computations may be seen in the appendix. 
%Why doesn't changing theta(p) impact the calculations? Not looking for a full computational answer here, but just some intuition would be welcome if there is any.

Applying Formula \ref{vertexratio}, we have 

\[
v_1(p) =  \csc \left(\frac{2(p + k_1)\pi}{pd} - \theta(p) \right) \left( \left(1 - v_0(p) \right)\sin \left(\theta(p) \right) - 
 \frac{\cos \left(\theta(p) - \frac{\pi}{p} \right) - \cos \left(\frac{ \left(2(p + k_1) - d \right)\pi}{dp} - \theta(p) \right)} {2\sin \left(\frac{\pi}{p} \right)} \right)
\]

Then,

\[v_1^* = \lim_{p \to \infty} \csc \left(\frac{2(p + k_1)\pi}{pd} - \theta(p) \right) \left( \left(1 - v_0(p) \right)\sin \left(\theta(p) \right) - \frac{\cos \left(\theta(p) - \frac{\pi}{p} \right) - \cos \left(\frac{ \left(2(p + k_1) - d \right)\pi}{dp} - \theta(p) \right)} {2\sin \left(\frac{\pi}{p} \right)} \right)  \]

\[= 1 - v_0^* - \lim_{p \to \infty} \frac{\theta^\prime(p)*p^2}{\pi} - \frac{k_1}{d}\] 

Now we show $\lim_{p \to \infty} \theta ^\prime (p)*p^2$ exists by computing it: 

\[\theta(p) = \arctan \left(\frac{\sum\limits_{i = 1}^n (-1)^{i + 1} \sin \left( \frac{2\pi S_i}{p} \right) }{\sum\limits_{i = 1}^n (-1)^{i + 1} \cos \left( \frac{2\pi S_i}{p} \right) } \right) + \frac{\pi}{2},\]
\[S_i = \mychi_{2\mathbb{Z} + 1}(i) \frac{p}{d} + \sum\limits_{j = 1}^i (-1)^{j + 1} \frac{k_j}{d} \Rightarrow\]

\[
\Rightarrow \lim_{p \to \infty} \theta ^\prime (p)*p^2 =
  \frac{-2\pi \sum\limits_{j = 1}^n (-1)^{j+1}(n - j + 1)(k_j)}{dn} \]

Thus we may conclude that $v_1^*$ exists and 

\[v_1^* = 1 - v_0^* - \frac{k_1}{d} + 2\sum\limits_{j = 1}^n \frac{(-1)^{j+1}(n - j + 1)(k_j)}{dn}\]

and by induction conclude that $v_i^*$ exists for all $i$ and

\[v_{i + 1}^* = 1 - v_i^* - \frac{k_{i + 1}}{d} + 2\sum\limits_{j = 1}^n \frac{(-1)^{j+1}(n - j + 1)(k_{j + i})}{dn}\]
\end{proof}

We now build the machinery to justify the step of allowing $v_0(p) = 0.5$.

\begin{lemma} 
\label{0.5}
It is sufficient to prove $v_i^* = 0$ or $1$ for some $i$ with $v_0(p) = 0.5$.
\end{lemma}

\begin{proof}
Diana Davis et al. \cite{davis} prove that all periodic billiard trajectories on the pentagon have reflectional symmetry, and their argument extends to the general odd-gon and thus the general prime-gon. Thus, for any closed geodesic on $X_p$, we can allow the development to have symmetry about a rotation of $180$ degrees, with a palindromic skip number sequence (that is, $s_i = s_{n - i}$ for a period-$n$ geodesic). 

Choose a starting point on a closed geodesic $\gamma$ that yields a palindromic skip number sequence, and let the angle of reflection from the starting edge $e$ of $X_p$ be $\theta$. Then a line $l$ passing through the midpoint of $e$ with angle of inclination $\theta$ would lie entirely inside the development of $\gamma$. This is because $l$ lies between $\gamma$ and $\gamma^\prime$, the image of $\gamma$ when we rotate the development $180$ degrees about its center. See Figure \ref{development} for an illustration of this proof. 

\begin{figure}
\includegraphics[scale = 0.35]{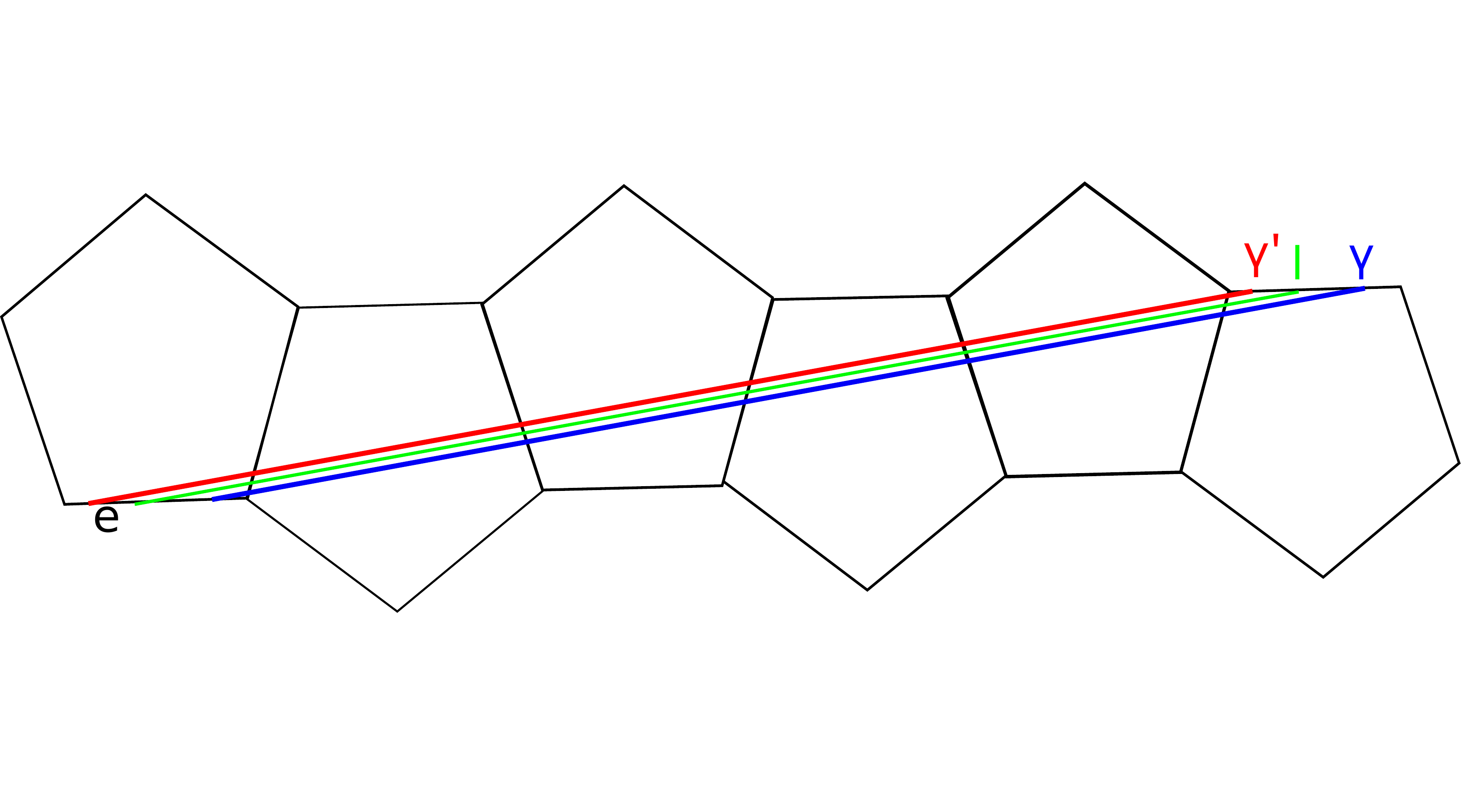}
\caption{A Visualization of Translating a Geodesic in its Development} 
\label{development}
\end{figure}

%If $l$ did not lie in the development, $l$ would have to be translated to pass through an interior edge $a$ of the development in which case it would not pass through the image $a^\prime$ of $a$ under rotation by $180^o$, so $\gamma$ could not have been a valid closed geodesic (see Figure that I will put in).

Thus, a closed geodesic with starting point at a particular edge of $X_p$ can be translated to pass through the midpoint of that edge if and only if the skip number sequence is palindromic, where the translation preserves the skip number sequence. Because every closed geodesic on $X_p$ can be translated to pass through a midpoint, we can assume that the skip number sequence $\{\frac{p + k_i}{d}\}_{i = 1}^n$ of $\gamma_p$ is rotated to be palindromic ($k_i = k_{n - i + 1}$) and that $\gamma_p$ is translated to pass through a midpoint with $v_0(p) = 0.5$.

In this case, $v_i(p) = 1 - v_{n - i}(p)$ for all $i$ by the symmetry of the closed geodesic. Thus, if we prove $v_i^* = 0$ for some $i$ in the particular case where $v_0(p) = 0.5$, $v_{n-i}^* = 1 - v_i^* = 1$. Now let $\gamma_i^\prime$ be a translated version of $\gamma_i$. Let $w_i(p)$ be the sequence of vertex ratios for $\gamma_i^\prime$, with $w_0(p) > v_0(p) = 0.5$ or $w_0(p) < v_0(p) = 0.5$. For each $p$, either $w_i(p) > v_i(p)$ and $w_{n - i}(p) > v_{n - i}(p)$ or $w_i(p) < v_i(p)$ and $w_{n - i}(p) < v_{n - i}(p)$, depending on the direction of translation of $\gamma_p$ to $\gamma_p^\prime$ (this is determined by whether $w_0(p) > 0.5$ or $w_0(p) < 0.5$).
This can be seen by traversing $\gamma_i^\prime$ in the reverse direction and noting that the $i^{th}$ vertex ratio encountered in this clockwise traversal is one minus the $(n - i)^{th}$ vertex ratio encountered in the counterclockwise one, where a vertex ratio in the clockwise direction is defined to be $\frac{b}{a + b}$ rather than $\frac{a}{a + b}$ (preserving the geometry of the notion in this dual correspondence). Then, $w_i(p) > v_i(p) \Rightarrow w_{n - i}(p) > v_{n - i}(p)$ and $w_i(p) < v_i(p) \Rightarrow w_{n - i}(p) < v_{n - i}(p)$ follow from reflection symmetry, since the $i^{th}$ vertex ratio in the clockwise orientation decreases when the $i^{th}$ vertex ratio in the counterclockwise orientation increases, and vice versa, depending on the direction of translation. 

%This can be seen by following the translation of the geodesic across the development.

Thus still either $w_i^* = 0$ or $w_{n - i}^* = 1$, up to a subsequence of the convergent geodesics where the direction of translation is the same. Note that the same reasoning applies if $v_i^* = 1$ by interchanging $i$ and $n - i$.

Breaking the sequence $\gamma_p^\prime$ into two subsequences, one where $w_0(p) < 0.5$ and one where $w_0(p) > 0.5$, there is still a sequence of vertices $v_p$ of $X_p$ where the distance of $\gamma_i^\prime$ from $v_p$ tends towards 0 upon fixing the edge length, and so the proof of Lemma \ref{converge} still applies. 
\end{proof}

\subsection{Properties of Vertex Ratio Limits}

We thus focus on proving $v_i^* = 0$ or $1$ for some $i$, with $v_0(p) = 0.5$ for all $p$ and a palindromic skip number sequence $\{\frac{p + k_i}{d}\}_{i = 1}^n$. The following corollaries follow from Proposition \ref{lims} and will be used in the proof.

\begin{Corollary} 
\label{3lims}
For all $i$,
$$v_{i - 1}^* + v_{i + 1}^* = 2(1 - v_i^*) + \frac{k_i - k_{i + 1}}{d}$$
\end{Corollary}

Corollary \ref{3lims} establishes a relationship between consecutive vertex ratio limits. In the arguments to come, it will be invaluable in deducing the patterns a skip number sequence exhibits. Namely, it will allow us to prove a result concerning the appearances of $v_i^* = 0.5$ as a vertex ratio limit. The following corollary of Proposition \ref{lims} also aids in this effort,  establishing an arithmetic relationship between alternating skip numbers. 

\begin{Corollary}
\label{arithmetic}
$v_0^*, v_2^*, v_4^*, \ldots, v_n^*$ and $v_1^*, v_3^*, v_5^*, \ldots, v_{n + 1}^*$ form arithmetic sequences when identified in the group $\mathbb{R}/\mathbb{Z}$ (in the group $\mathbb{Q}/\mathbb{Z}$, in fact, for it is clear from Proposition \ref{lims} that if $v_0^* = 0.5$ then all $v_i^*$ are rational) . That is, $v_{i + 2}^* - v_i^* = v_i^* - v_{i - 2}^*$ in $\mathbb{R}/\mathbb{Z}$ (or equivalently, $v_{i + 2}^* - v_i^* \equiv v_i^* - v_{i - 2}^* (mod\ 1))$  for all $i$.
\end{Corollary}

We now work our way towards a proof of $v_i^* = 0$ or $1$ for some $i$ by contradiction. We make extensive use of the above two corollaries. The final piece of machinery we present before the proof is a restriction concerning the presence of $0.5$'s in the $v_i^*$ sequence, assuming for the sake of contradiction that $v_i^* \neq 0$ nor $1$ for all $i$. 

\begin{prop}
\label{periodicity}
If $0 < v_i^* < 1$ for all $i \in \mathbb{Z}$, then all occurrences of $0.5$ in the sequence $\ldots, v_{-2}^*, v_{-1}^*, v_0^*, v_1^*, v_2^*, \ldots$ are periodic with odd period. That is, $\{i \in \mathbb{Z} \mid v_i^* = 0.5\} = t\mathbb{Z}$ for some $t \in 2\mathbb{Z} + 1$.
\end{prop}

\begin{proof}
By Corollary \ref{3lims}, we have 
 
\[v_{i - 1}^* + v_{i + 1}^* = 1 \Leftrightarrow v_i^* = 0.5\]

This is because $\frac{k_i - k_{i+1}}{d} \in \{-1, 0 ,1\}$ by Proposition \ref{skips} and because $0 < v_{i - 1}^*, v_i^*, v_{i + 1}^* < 1$. It follows that for all $k$

\begin{equation}\label{eqn:0.5consecs}
    v_{i - k}^* + v_{i + k}^* = 1 \Leftrightarrow v_i^* = 0.5
\end{equation}

This is because if $k$ is odd, $v_{i - 1}^* - v_{i - k}^* = v_{i + k}^* - v_{i + 1}^*$ in $\mathbb{R}/\mathbb{Z}$ by Corollary \ref{arithmetic}, so $v_{i - k}^* + v_{i + k}^* = 1 \Leftrightarrow v_{i - 1}^* + v_{i + 1}^* = 1$ since $0 < v_{i- k}^*, v_{i - 1}^*, v_{i + 1}^*, v_{i + k}^* < 1$. And if $k$ is even, $v_{i - k}^*, v_i^* , v_{i + k}^*$ form an arithmetic progression in $\mathbb{R}/\mathbb{Z}$ by Corollary \ref{arithmetic} so $v_{i - k}^* + v_{i + k}^* = 1 \Leftrightarrow 2v_i^* = 1 \Leftrightarrow v_i^* = 0.5$ since $0 < v_{i - k}^*, v_i^*, v_{i + k}^* < 1$.

Now let $\ldots ,v_{i_{-2}}^*, v_{i_{-1}}^*, v_{i_0}^*, 
v_{i_1}^*, v_{i_2}^*, \ldots$ be the subsequence of 0.5 occurrences in the $v_i^*$ sequence. The subsequence is infinite in both directions because $v_0^* = 0.5$ so $v_i^* = 0.5$ for all $i \in n\mathbb{Z}$. Let $v_{i_{k-1}}^*, v_{i_k}^*, v_{i_{k+1}}^*$ be three arbitrary, consecutive elements in the subsequence. 

Assume first that $i_k - i_{k-1} \neq i_{k + 1} - i_k$ so that without loss of generality $i_k - i_{k-1} < i_{k + 1} - i_k$. 
Then by (\ref{eqn:0.5consecs}),  $v_{i_{k - 1}}^*, v_{i_k}^*  = 0.5 \Rightarrow v_{2i_k - i_{k - 1}}^* = 0.5$ since $2i_k - i_{k - 1} = i_k + (i_k - i_{k - 1})$. Thus we have that $i_k - i_{k-1} < i_{k + 1} - i_k \Rightarrow i_k < 2i_k - i_{k - 1} < i_{k + 1}$ with $v_{2i_k - i_{k - 1}}^* = 0.5$, a contradiction since $v_{i_{k-1}}^*, v_{i_k}^*,$ and  $v_{i_{k+1}}^*$ were consecutive (in the case where $i_k - i_{k-1} > i_{k + 1} - i_k$, the proof is analogous with $v_{2i_k - i_{k + 1}}^* = 0.5$ and $i_{k - 1} < 2i_k - i_{k + 1} < i_k$).

Hence $i_k - i_{k-1} = i_{k + 1} - i_k$, implying that occurrences of 0.5 are periodic. Now if $t = i_k - i_{k-1} = i_{k + 1} - i_k$ were even, $m = \frac{i_k + i_{k+1}}{2} \in \mathbb{Z}$ and $v_{i_k}^*, v_{i_{k + 1}}^* = 0.5 \Rightarrow v_m^* = 0.5$ by (\ref{eqn:0.5consecs}), a contradiction since $i_k < m < i_{k+1}$ but $v_{i_{k-1}}^*, v_{i_k}^*,$ and $v_{i_{k+1}}^*$ are consecutive. Thus, occurrences of 0.5 are periodic with period $t$ odd.
\end{proof} 

\begin{rema}
Since $v_{cn}^* = v_0^*$ for all $c \in \mathbb{Z}$, we know $t \mid n$. Since $t$ is odd, we can deduce $t \mid \frac{n}{2^e}$, where $2^e$ is the highest power of $2$ dividing $n$.
\end{rema}

\subsection{Proof that the Minimizing Index Tends to Infinity}

We have now developed the machinery needed to complete the last step in the proof that $minind(X_p) \rightarrow \infty$:

\begin{prop}  $v_i^* = 0$ or $v_i^* = 1$ for some $i \in \mathbb{Z}$.
\label{Lim0}
\end{prop}
%Some sort of prose in this proof would be nice. I don't know how useful diagrams would be but something to break up all the algebra (even just a brief description of what you're doing at a high level) would be nice.

\begin{proof} 
Assume for the sake of contradiction that $0 < v_i^* < 1$ for all $i \in \mathbb{Z}$. Then, by Proposition \ref{periodicity}, occurrences of 0.5 in the sequence $\ldots, v_{-2}^*, v_{-1}^*, v_0^*, v_1^*, v_2^*, \ldots$ are periodic with odd period $t$. 

The idea of this proof is to show through casework that the sequence $v_0^*, 1 - v_1^*, v_2^*, 1- v_3^*, \ldots, 1 - v_t^*, \ldots, v_{2t}^*$ is $t$-periodic with $t \mid 2n$. We may then use this to deduce that the sum of the skip numbers $lp$ is $n$ multiplied by some fixed integer with $l < n$, and this cannot be true for a sequence of primes $p$.

$v_0^*, v_2^*, v_4^*, \ldots v_{2t}^*$ form an arithmetic sequence in $\mathbb{R}/\mathbb{Z}$ by Corollary \ref{arithmetic}, with $v_0^* = v_{2t}^* = 0.5$.
Thus, the common difference must be $\frac{k}{t}$ where $k \in \mathbb{Z}$. Without loss of generality we can assume $0 \leq k < t$. Then, $v_{2i}^* = 0.5 + \frac{ik}{t} = \frac{t + 2ik}{2t}$ in $\mathbb{R}/\mathbb{Z}$, and since $t$ is odd $t + 2ik$ is odd so $v_{2i} = \frac{e_{2i}}{2t}$ for $e_{2i} \in \{1, 3, 5, \ldots, t, \ldots, 2t - 1\}$. The same reasoning applies to $v_1^*, v_3^*, v_5^*, \ldots, v_t^*, \ldots, v_{2t - 1}^*$ with $v_t^* = 0.5$ implying that for all $i$, $v_i^* = \frac{e_i}{2t}$ for $e_i \in \{1, 3, 5, \ldots, t, \ldots, 2t - 1\}$.

Now, assume without loss of generality that the common difference $\frac{k}{t} = v_2^* - v_0^*$ is with $k < \frac{t}{2}$ (if $k > \frac{t}{2}$, we may traverse the geodesics in the reverse direction, reversing the sequence with $v_2^* - v_0^* = \frac{t - k}{t}$ with $t - k < \frac{t}{2}$). Then, $v_{t - 1}^* = 0.5 + \frac{t - 1}{2}*\frac{k}{t} = \frac{t + tk - k}{2t}$ in $\mathbb{R}/\mathbb{Z}$.

In the case that $k$ is even, $v_{t - 1}^* =  \frac{t + tk - k}{2t} = \frac{t - k}{2t}$ in $\mathbb{R}/\mathbb{Z}$. Since $v_t^* = 0.5$, (\ref{eqn:0.5consecs}) implies $v_{t + 1}^* = \frac{t + k}{2t}$ in $\mathbb{R}/\mathbb{Z}$. Also, $v_0^* = 0.5, v_2^* = \frac{t + 2k}{2t} \Rightarrow 1 - v_1^* = \frac{t + k}{2t}$ by Corollary \ref{3lims}, for $\frac{k_i - k_{i + 1}}{d} \in \{-1, 0, 1\}$ and $\frac{k_i - k_{i + 1}}{d} = 0$ is the only choice that ensures $v_1^* = \frac{e_1}{2t}$ with $e_1 \in \{1, 3, 5, \ldots, t, \ldots, 2t - 1\}$ as required. Thus $1 - v_1^* = v_{t + 1}^*$, and together with Corollary \ref{arithmetic} and $t$-periodicity of 0.5 occurrences, this implies $t$-periodicity of the sequence $\ldots, 1- v_{-3}^*, v_{-2}^*, 1 - v_{-1}^*, v_0^*, 1- v_1^*, v_2^*, 1 - v_3^*, \ldots$. 

In the case that $k$ is odd, $v_{t - 1}^* =  \frac{t + tk - k}{2t} = \frac{2t - k}{2t}$ in $\mathbb{R}/\mathbb{Z}$. Since $v_t^* = 0.5$, (\ref{eqn:0.5consecs}) implies $v_{t + 1}^* = \frac{k}{2t}$ in $\mathbb{R}/\mathbb{Z}$. Also, $v_0^* = 0.5, v_2^* = \frac{t + 2k}{2t} \Rightarrow 1- v_1^* = \frac{k}{2t}$ by Corollary \ref{3lims}, for $\frac{k_i - k_{i + 1}}{d} \in \{-1, 0, 1\}$, and $\frac{k_i - k_{i + 1}}{d} = -1$ is the only choice that ensures $v_1^* = \frac{e_1}{2t}$ with $e_1 \in \{1, 3, 5, \ldots, t, \ldots, 2t - 1\}$ and $0 < v_1^* < 1$ as required (since $k < \frac{t}{2}$). Thus $1 - v_1^* = v_{t + 1}^*$, and together with Corollary \ref{arithmetic} and $t$-periodicity of 0.5 occurrences, this again implies $t$-periodicity of the sequence $\ldots, 1- v_{-3}^*, v_{-2}^*, 1 - v_{-1}^*, v_0^*, 1- v_1^*, v_2^*, 1 - v_3^*, \ldots$ (Note that $1 - v_{jt}^* = v_{jt}^*$ for all $j \in \mathbb{Z}$ in this sequence since $v_{jt}^* = 0.5$). This also implies $v_i^*= v_{i + 2t}^*$ for all $i \in \mathbb{Z}$ so that $\ldots, v_{-2}^*, v_{-1}^*, v_0^*, v_1^*, v_2^*, \ldots$ is $2t$-periodic (Note that Corollary \ref{arithmetic} applied to the sequence $v_0^*, v_2^*, v_4^*, \ldots v_{2t}^*$ also implies this because of $t$-periodicity of 0.5 occurences).

In either case, the sequence $v_0^*, 1 - v_1^*, v_2^*, 1- v_3^*, \ldots, 1 - v_t^*, \ldots, v_{2t}^*$ is $t$-periodic with $2t \mid n$ since $t$ is odd, $n$ is even, and $t \mid n$. 
%Do we ever prove/justify why n is even? It's not hard (just alternating sides argument) but might be worth putting in for intuition. In fact, how do you define n? Do we ever give intuition for what n is?
Since $v_t^* = 0.5$, by Corollary \ref{arithmetic} and (\ref{eqn:0.5consecs}), $v_{t - i}^* + v_{t + i}^* = 1$ for all $i$. Then $t$-periodicity of the sequence above implies $1- v_{t + i}^* = v_i^* \Rightarrow v_i^* = v_{t - i}^*$ and $v_{t + i}^* = v_{2t - i}^*$ for all $i \in \mathbb{Z}$. Thus, given the sequence of skip numbers $s_1 = \frac{p + k_1}{d}, s_2 = \frac{p + k_2}{d}, s_3 = \frac{p + k_3}{d}, \ldots, s_n = \frac{p + k_n}{d}$, Corollary \ref{3lims} implies $s_{i + 1} - s_i = s_{t + 1 - i} - s_{t - i}$ and $s_{t + i + 1} - s_{t + i} = s_{2t - i + 1} - s_{2t - i}$ for all $i \in \mathbb{Z}$ with $1 \leq i < t$. Thus, we have

\[\sum\limits_{j = 1}^{2t} s_j = \sum\limits_{j = 1}^{2t} \left(s_1 + \sum\limits_{k = 1}^{j - 1} \left(s_{k + 1} - s_k\right)\right) =  2ts_1 + \sum\limits_{j = 1}^{2t - 1} (2t - j) \left(s_{j + 1} - s_j \right) = \]

\[= 2ts_1 + \sum\limits_{j = 1}^{\frac{t - 1}{2}} \Big(\big[(2t - j) + \left(2t - (t - j)\right)\big] \left(s_{j + 1} - s_j \right) \Big) + t(s_{t + 1} - s_t) + \]
\[\sum\limits_{j = t + 1}^{\frac{3t - 1}{2}} \Big(\big[(2t - j) + \left(2t - (2t - (j - t))\right)\big] \left(s_{j + 1} - s_j \right) \Big) =\]

\[= 2ts_1 + \sum\limits_{j = 1}^{\frac{t - 1}{2}} \Big(3t\left(s_{j + 1} - s_j \right) \Big) + t(s_{t + 1} - s_t) + \sum\limits_{j = t + 1}^{\frac{3t - 1}{2}} \Big(t \left(s_{j + 1} - s_j \right) \Big)  \]

\[= 2ts_1 + \sum\limits_{j = 1}^{\frac{t - 1}{2}} \Big(3t\left(s_{j + 1} - s_j \right) \Big) + \sum\limits_{j = t + 1}^{\frac{3t - 1}{2}} \Big(t \left(s_{j + 1} - s_j \right) \Big)\]

%$$= 2ts_1 + (2t - 1)(s_2 - s_1) + (2t - 2)(s_3 - s_2) + \ldots + (s_{2t} - s_{2t - 1}) = 2ts_1 + (2t - 1 + 2t - (t - 1))(s_2 - s_1) + (2t - 2 + 2t - (t - 2))(s_3 - s_2) + $$

%$$+ (2t - 3 + 2t - (t - 3))(s_4 - s_3) + \ldots + (2t - \frac{t - 1}{2} + 2t - \frac{t + 1}{2})(s_{(t + 1)/2} - s_{(t - 1)/2}) + t(s_{t + 1} - s_t) + (2t - (t + 1) + 2t - (2t - 1))(s_{t + 2} - s_{t + 1}) +$$

 %$$+ (2t - (t + 2) + 2t - (2t - 2))(s_{t + 3} - s_{t + 2}) +  (2t - (t + 3) + 2t - (2t - 3))(s_{t + 4} - s_{t + 3}) + \ldots +  (2t - \frac{3t - 1}{2} + 2t - \frac{3t +1}{2}))(s_{(3t + 1)/2} - s_{(3t - 1)/2}) = $$
 
 %$$= 2ts_1 + 3t(s_2 - s_1) + 3t(s_3 - s_2) + 3t(s_4 - s_3) + \ldots + 3t(s_{(t + 1)/2} - s_{(t - 1)/2}) + t(s_{t + 1} - s_t) + $$
 
% $$t(s_{t + 2} - s_{t + 1}) + t(s_{t + 3} - s_{t + 2}) +  t(s_{t + 4} - s_{t + 3}) + \ldots +  t(s_{(3t + 1)/2} - s_{(3t - 1)/2})$$
 
 %$$ = 2ts_1 - 3ts_1 + 3ts_{(t + 1)/2}  - ts_t + ts_{(3t + 1)/2}$$

where the last step follows from $s_{t + 1} - s_t = 0$ by Corollary \ref{3lims} since $v_t^* = 0.5$ and $v_{t - 1}^* + v_{t + 1}^* = 1$. Now, because the sequence $\ldots, 1- v_{-3}^*, v_{-2}^*, 1 - v_{-1}^*, v_0^*, 1- v_1^*, v_2^*, 1 - v_3^*, \ldots$ is $t$-periodic, we have $v_i^* = 1 - v_{i + t}^*$ for all $i \in \mathbb{Z}$. Applying Corollary \ref{3lims}, we have

\[v_{i - 1}^* + v_{i + 1}^* - 2(1 - v_i^*) = \frac{k_i - k_{i + 1}}{d} \Rightarrow\]

\[1 - v_{i + t - 1}^* + 1 - v_{i + t + 1}^* - 2v_{i + t}^* = \frac{k_i - k_{i + 1}}{d} \Rightarrow\]

\[- \left(v_{i + t - 1}^* + v_{i + t + 1}^* - 2(1 - v_{i + t}^*) \right) = \frac{k_i - k_{i + 1}}{d} \Rightarrow\]

\[- \frac{k_{i + t} - k_{i + t + 1}}{d} = \frac{k_i - k_{i + 1}}{d} \Rightarrow\]

\[s_{i + 1} - s_i = -(s_{t + i + 1} - s_{t + i})\]

Thus, $s_{2t + 1} = s_1 + \sum\limits_{j = 1}^{2t} \left(s_{j + 1} - s_j\right) = s_1 + \sum\limits_{j = 1}^{t} \left(s_{j + 1} - s_j\right) + \sum\limits_{j = t + 1}^{2t} \left(s_{j + 1} - s_j\right) = s_1 + \sum\limits_{j = 1}^{t} \left(s_{j + 1} - s_j\right) - \sum\limits_{j = 1}^{t} \left(s_{j + 1} - s_j\right) = s_1$. Since the sequence $v_0^*, v_1^*, v_2^*, \ldots$ is $2t$-periodic and $s_1 = s_{2t + 1}$, the sequence of skip numbers is $2t$-periodic since the $v_0^*, v_1^*, v_2^*, \ldots$ sequence uniquely determines differences between consecutive skip numbers by Corollary \ref{3lims}. Thus, 
\begin{equation}\label{eqn:skipsums}
\sum\limits_{j = 1}^n s_j = \frac{n}{2t} \sum\limits_{j = 1}^{2t} s_j
\end{equation} 
(it was determined previously that $2t \mid n$). Also $s_{i + 1} - s_i = -(s_{t + i + 1} - s_{t + i})$ now implies 

\[\sum\limits_{j = 1}^{2t} s_j = 2ts_1 + \sum\limits_{j = 1}^{\frac{t - 1}{2}} \Big(3t\left(s_{j + 1} - s_j \right) \Big)  + \sum\limits_{j = t + 1}^{\frac{3t - 1}{2}} \Big(t \left(s_{j + 1} - s_j \right) \Big) =\]

\[= 2ts_1 + \sum\limits_{j = 1}^{\frac{t - 1}{2}} \Big(3t\left(s_{j + 1} - s_j \right) \Big) - \sum\limits_{j = 1}^{\frac{t - 1}{2}} \Big(t \left(s_{j + 1} - s_j \right) \Big) = \]

\[ = 2ts_1 + \sum\limits_{j = 1}^{\frac{t - 1}{2}} \Big(2t\left(s_{j + 1} - s_j \right) \Big) =\]
 
\[= 2ts_{(t + 1)/2}\]

Recalling that the sum of the skip numbers of $\gamma_p$ is $lp$ for $l < n$, by (\ref{eqn:skipsums}), $lp = \sum\limits_{j = 1}^{n} s_j = ns_{(t + 1)/2} \Rightarrow \frac{lp}{n} = s_{(t + 1)/2}$ for all primes $p$ in the sequence $\gamma_p$. 
Since $s_{(t + 1)/2} \in \mathbb{Z}^+$, we must have $n \mid lp$.  However, $l < n$ and all the $p$ are prime, which implies that $n \mid p$ for all $p$ in the sequence $\gamma_p$.  
Since our sequence $\gamma_p$ is infinite and $n$ is a fixed finite integer, this is a contradiction.  Thus $v_i^* = 0$ or $v_i^* = 1$ for some $i \in \mathbb{Z}$.
%But since $l < n$ and $s_{(t + 1)/2} \in \mathbb{Z}^+$, $n$ must have a prime factor $n^\prime$ that $l$ does not, so $n^\prime \mid p$ for all primes $p$ in the sequence. This is a contradiction, so $v_i^* = 0$ or $v_i^* = 1$ for some $i \in \mathbb{Z}$.
%Technicality: l and n could have the same prime factors but n could have one of them to a higher power. Proof still works but this is worth noting, and I'm not sure the best way to phrase this.
%@ARTHUR: I rewrote the ending to this proof to be a little more clear and to accommodate my technicality above, and also amended the remark below to match the style. Does this look good?
\end{proof}

\begin{rema}
Note that the only time primality was used in this whole chain of reasoning was at the very end. If instead we have a sequence $n_i \in \mathbb{N}$, the reasoning can be modified to obtain a contradiction if no subsequence of $n_i$ has a common prime factor. Thus, for a sequence $n_i \in \mathbb{N}$ the statement holds if no subsequence $n_{i_j}$ has $gcd(n_{i_j}) > 1$. 

%We note that we can modify the conclusion $n \mid p$ above to $n \mid p_i$, where $p_i$ is a prime factor of $n_i$ for some sequence $n_i$. This can be used to show that the statement holds for any sequence $n_i \in \mathbb{N}$ such that no subsequence $n_{i_j}$ has $gcd(n_{i_j}) > 1$.
%It can thus (with the other parts of the whole proof) be shown that $mindind(X_{n_i}) \rightarrow \infty$ for such $n_i$. 
\end{rema}

Theorem \ref{infinity} now follows immediately from Proposition \ref{Lim0}, Lemma \ref{converge} and Lemma \ref{convergent}. The generalization of Theorem \ref{infinity} below also follows from the above remark.

%\begin{theo}
%\label{infinity}
%$minind(X_p) \rightarrow \infty$
%\end{theo}

%\begin{proof}
%This now follows immediately from Theorem \ref{Lim0}, Lemma \ref{converge} and Lemma \ref{convergent}.
%For every convergent sequence of geodesics on some subsequence of $X_p$, we have extracted a subsequence with skip numbers $\frac{p + k_1}{d}, \frac{p + k_2}{d}, \frac{p + k_3}{d}, \ldots, \frac{p + k_n}{d}$ where the sum of the skip numbers in the sequence is $lp$ for $l \in \mathbb{N}$ with $l < n$ and constant $k_i$.

%By Theorem \ref{Lim0} and Lemma \ref{converge}, the geodesics in this sequence have minimizing indices converging to $\infty$. Thus, for every convergent sequence of geodesics on some subsequence of $X_p$, the minimizing indices of the geodesics are unbounded. Thus, by Lemma \ref{convergent}, $minind(X_p) \rightarrow \infty$.
%\end{proof}

\begin{theo}
\label{generalization}
For all $n_i \in \mathbb{N}$, $minind(X_{n_i}) \rightarrow \infty$ if and only if $gcd(n_{i_j}) = 1$ for all subsequences $n_{i_j}$. 
\end{theo}

\begin{proof} 
For the forward direction, if there is a subsequence $n_{i_j}$ where $gcd(n_{i_j}) \neq 1$, we would have a contradiction of either \cite[Proposition 2.1]{ade} or Corollary \ref{infinity} which implies $minind(X_{n_{i_j}}) \leq 2gcd(n_{i_j})$. For the reverse direction, we may apply the same proof as for the primes.
\end{proof}

\section{Convergent V-shaped Geodesics} \label{vshapesec}

In this section we discuss the V-shaped geodesics on the doubled odd-gons from Example~\ref{vshapeEx} (see Figure~\ref{vshape}) and prove Theorem~\ref{vshapethm}. 

We first compute the minimizing index of the V-shaped geodesics, providing a counterexample to the converse of Sormani's Corollary 7.2 in the process.  
We then focus on proving Theorem~\ref{vshapethm}. In subsection 5.2, we provide and justify preliminary assumptions used in the proof of this theorem and establish the setup we will use in the proof. Using these assumptions and this setup, we will then provide short intuitive proofs of Theorem~\ref{vshapethm} for $X_3$ and $X_5$ and then generalize these proofs to $X_{2n + 1}$.

\subsection{Minimizing Index}
Recall that the $V$-shaped geodesic on $X_{2n + 1}$ starts at the midpoint $M$ of some edge and consists of perpendiculars to opposing edges. 
These geodesics exist for all $n$, have period 4, and have skip number sequence $\{n, n+1, n+1, n\}$.
Note that this sequence meets the conditions of Proposition $\ref{skips}$.
%This skip sequence probably doesn't need to exist anymore. I had this in orginally as an example of Prop 4.2 -- should it still be moved there or just compeltely removed?  Or should I leave it here so we can cite it in Section 4?
For the remainder of this paper, we denote by $g_1, g_2, g_3, g_4$ the consecutive segments of the $V$-shaped geodesic as traversed in a counterclockwise orientation (see Figure~\ref{VSP1}). 

\begin{figure}[ht]
\includegraphics[scale = 0.25]{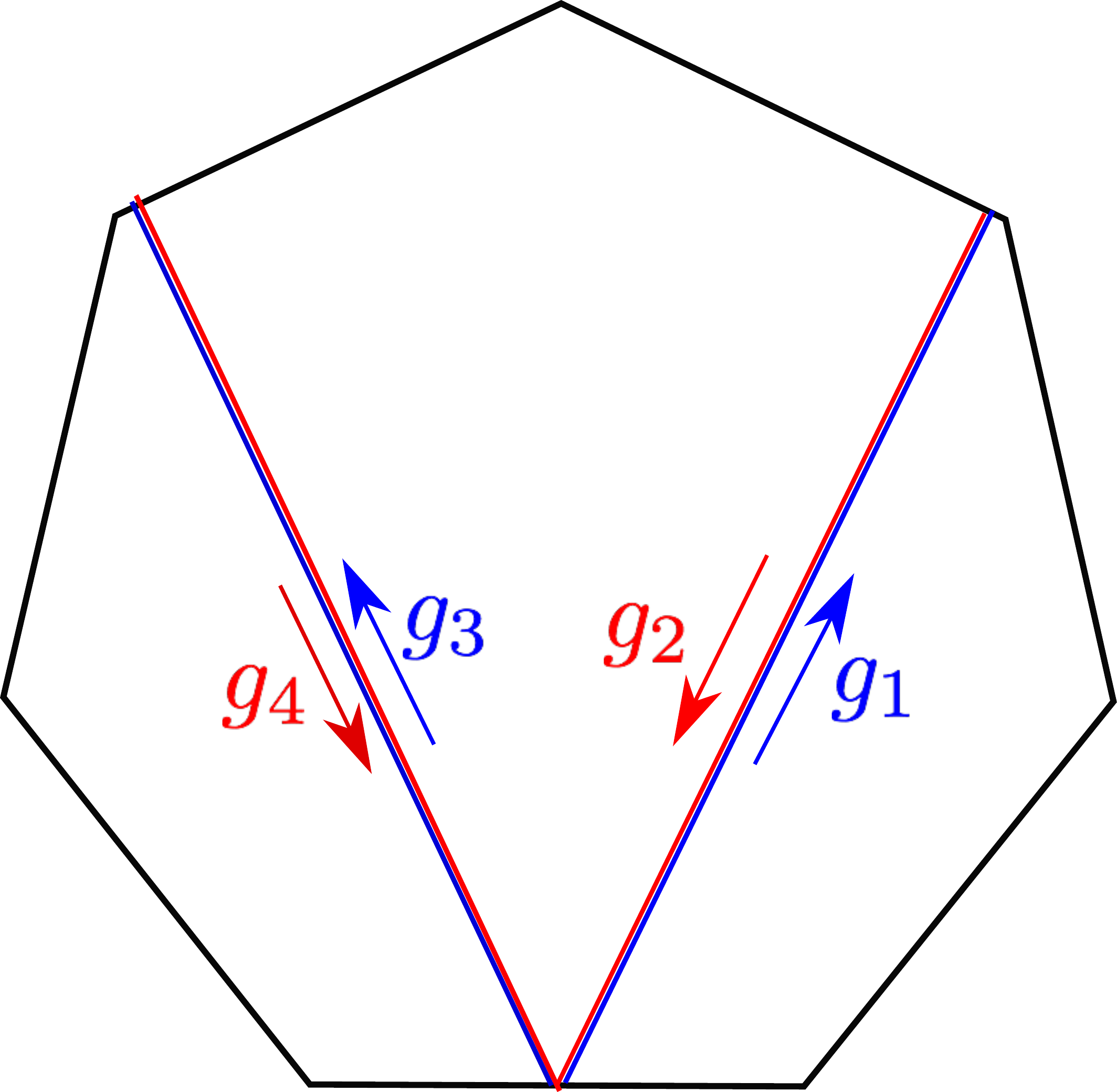}
\caption{Labeling Segments of V-Shaped Geodesic in a Counterclockwise traversal} 
\label{VSP1}
\end{figure}

\begin{figure}
\includegraphics[scale=0.3]{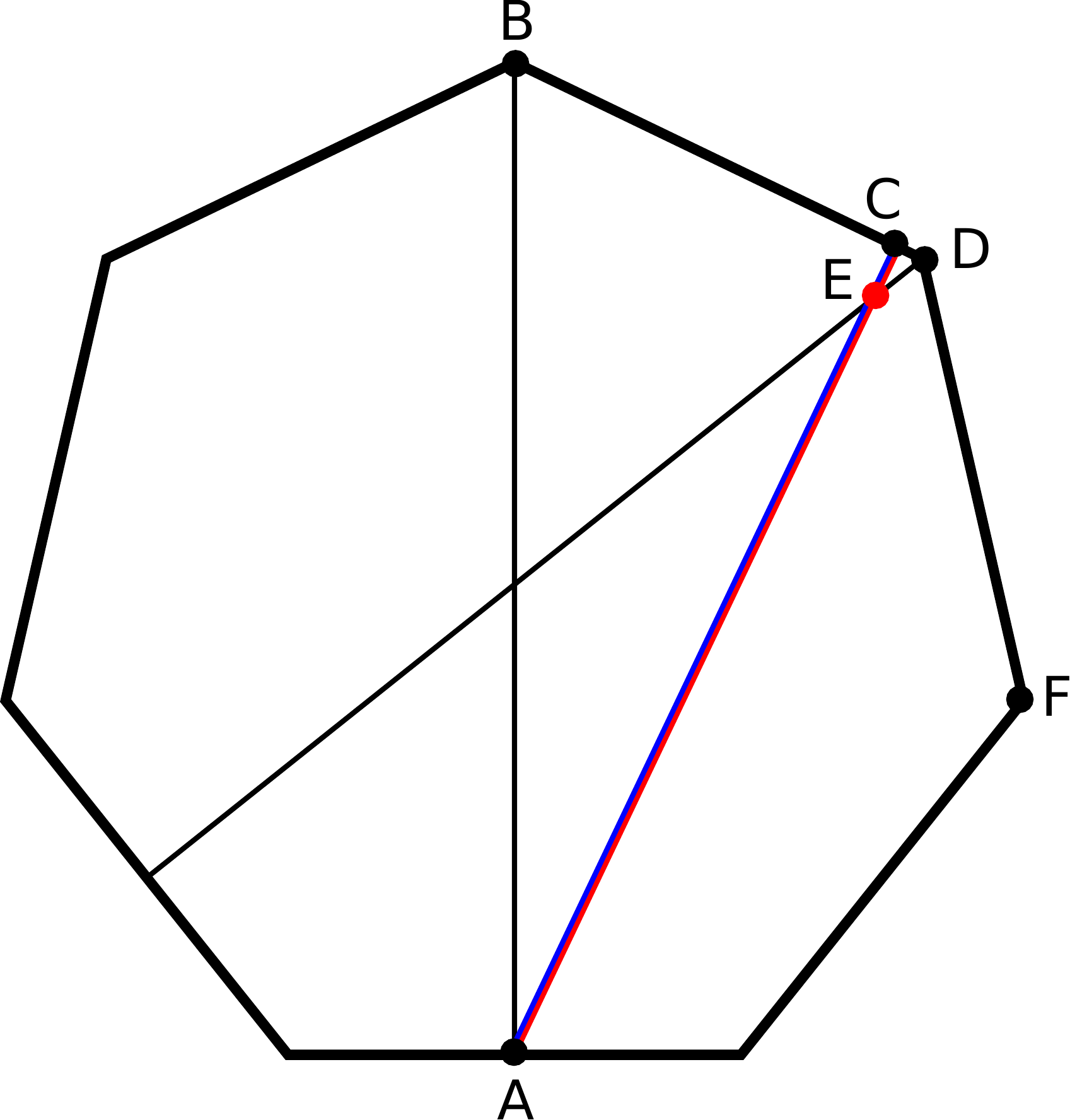}
\caption{Part of a V-shaped curve on a 7-gon}
\label{vshapepf}
\end{figure}

We now provide a bound for the minimizing index of this geodesic that is far better than its period.
In Figure $\ref{vshapepf}$, $A$ is the midpoint of a side of $X_{2n+1}$ and $AC$ is the path the V-shaped geodesic takes, hitting side $BD$ at a right angle.
Let $E$ be the point of intersection of the angle bisector of $D$ with $AC$.
We claim that the path along the geodesic from $E$ to $C$ and then to $E$ on the other face of the polygon is the maximal length-minimizing path for all intervals of $\gamma$ centered at $C$.
This is because $E$ is the intersection of the angle bisector at $D$ and a perpendicular from side $BD$, so it is equidistant to $BD$ and $DF$. 
If $E$ were moved any closer to $A$, it would be closer to $DF$ and a geodesic hitting side $DF$ would be a shorter way to connect $E$ with its corresponding point on the opposite face of $X_{2n+1}$. 

Let $x$ be the length of $CE$. The length of $\gamma$ is $4l(g_1)$ (our geodesic is period 4 and all the segments have the same length by symmetry) and the maximal interval centered at $C$ that is length-minimal has length $2x$.
This gives the bound $minind(\gamma) \geq 4l(g_1)/2x = 2l(g_1)/x$.
Note that $\angle CBA = \angle CDE = \frac{(2n-1)\pi}{2(2n+1)}$ since $AB$ and $DE$ are angle bisectors and the polygon is regular.
Thus $\triangle ABC \sim \triangle EDC$. Note that $\angle CAB = \angle DEC = \frac{\pi}{2n+1}$.

For convenience, we let $X_{2n+1}$ have side length 1. Let $CD = y$. Then $BC = 1-y$. Further let $AB = h$.
%Is this too confusing to have inradius 1 in the following subsections and side length 1 now? It's all ratios so it doesn't matter for the calculations but I think we should redo one bit or the other.
Now note that $\sin(\frac{\pi}{2n+1}) = \frac{1-y}{h}$, so $y = 1 - h\sin(\frac{\pi}{2n+1})$.
By similar triangles, $\frac{l(g_1)}{x} = \frac{1-y}{y}$, so $minind(\gamma) \geq \frac{2(1-y)}{y} = \frac{2h\sin(\frac{\pi}{2n+1})}{1-h\sin(\frac{\pi}{2n+1})}$.
Note that $h$ is just the height of a regular $2n+1$-gon, which is equal to the sum of the circumradius and the apothem.
This comes out to $h = \frac{1+\cos(\frac{\pi}{2n+1})}{2\sin(\frac{\pi}{2n+1})}$.

Plugging this into our bound for the minimizing index gives $$minind(\gamma) \geq \frac{1+\cos(\frac{\pi}{2n+1})}{1-\frac{1}{2}(1-\cos(\frac{\pi}{2n+1}))} = \frac{1}{2}\left(\frac{1+\cos(\frac{\pi}{2n+1})}{1-\cos(\frac{\pi}{2n+1})}\right).$$

Note that as $n\to\infty$, this lower bound on the minimizing index tends to $\infty$ as well. 
However, in the limit, the angle the V-shape makes, $\frac{2\pi}{2n+1}$, tends to zero.
Thus the limit of this convergent sequence of geodesics is a geodesic on the double disk that traverses a diameter geodesic twice.
This is a period 4 geodesic on the double disk, so by Theorem~\ref{double disk minind}, it has minimizing index 4.
This provides a counterexample to the converse of Sormani's Corollary 7.2, as we have demonstrated a sequence of geodesics with divergent minimzing index but non-trivial Gromov-Hausdorff limit.

Furthermore, it is known that the only period 4 geodesic on the double triangle $X_3$ is this V-shaped one.  Plugging in $n=1$ to the above formula gives $minind(\gamma) \geq 6$, which proves that $minind(X_3) \geq 6$ as well (using Proposition \ref{minindperpf}).\newline

\subsection{Preliminaries}
Now we transition to proving Theorem $\ref{vshapethm}$.  By the same reasoning used in the proof of Lemma~\ref{0.5}, a closed geodesic on a $X_{2n + 1}$ can be translated to pass through the midpoint of some edge of the polygon. Since such a translation preserves the skip number sequence, it preserves geodesic length because it preserves the shape of the development. 

Thus, for the purposes of proving Theorem~ \ref{vshapethm}, we may assume without loss of generality that all geodesics start at the midpoint of some edge. 

As in the proof of Lemma \ref{0.5}, this assumption is equivalent to a palindromic skip number sequence and to reflectional symmetry of a geodesic.

Let $\gamma$ be an arbitrary closed geodesic on $X_{2n + 1}$, starting at $M$ and consisting of consecutive segments $\gamma_1, \gamma_2, \gamma_3, \ldots, \gamma_{2k}$ in a counterclockwise orientation of $\gamma$.
Due to reflectional symmetry, the midpoint of the geodesic, traversed between $\gamma_k$ and $\gamma_{k + 1}$, lies on the line $l$ that contains $M$ and is perpendicular to the edge containing $M$. Thus, because $2n + 1$ is odd, $l$ intersects the boundary of the polygon $X_{2n}$ at $M$ and at the vertex opposing $M$. Because a geodesic cannot pass through a vertex, the midpoint of the geodesic must coincide with the starting edge midpoint $M$.

Thus, $g_1 \rightarrow g_2$ and $\gamma_1 \rightarrow \gamma_2 \ldots \rightarrow \gamma_k$ are paths traversed by the $V$-shaped geodesic and $\gamma$, respectively, from $M$ back to $M$.
We want to show that the length of $\gamma$ is at least the length of the $V$-shaped geodesic on $X_{2n+1}$.
%this should probably have a uniqueness qualifier (strictly greater than unless gamma is the V-shape) but the rest of the paper isn't written that way
Because $\displaystyle \sum_{i = 1}^k l(\gamma_i) = \frac{l(\gamma)}{2}$ and $l(g_1) + l(g_2)$ is half the length of the $V$-shaped geodesic, it suffices to show 

\begin{equation}\label{eqn:vshapepf}
\displaystyle \sum_{i = 1}^k l(\gamma_i) \geq l(g_1) + l(g_2),
\end{equation}

with equality only when $\gamma$ is the V-shaped geodesic, in order to prove Theorem~ \ref{vshapethm}. The following proof of Theorem~ \ref{vshapethm} will be devoted to proving (\ref{eqn:vshapepf}). Due to reflectional symmetry of $\gamma$, we may assume without loss of generality that the skip number $s_1$ of $\gamma_1$ is less than or equal to $n$ (i.e. less than halfway across the circumference of $X_{2n + 1})$. 

\subsection{V-shaped geodesics on $X_3$ and $X_5$}
We now provide an intuitive proof of (\ref{eqn:vshapepf}) for $X_3$ and $X_5$.

On $X_3$, $s_1 \leq 1 \Rightarrow s_1 = 1$. Since the skip number of $g_1$ is also $1$, $\gamma_1$ hits the same edge (opposing $M$) of $X_3$ as $g_1$. (\ref{eqn:vshapepf}) then follows from the fact that a perpendicular is the shortest path from a point to a line: $\gamma_1$ is a path to the edge from $M$ and $\gamma_2 \rightarrow \ldots \rightarrow \gamma_k$ is a path from that edge to $M$. So $l(g_1) \leq l(\gamma_1)$ and $l(g_2) \leq \sum_{i = 2}^k l(\gamma_i)$ because $g_1$ and $g_2$ are perpendiculars.  Note that equality holds only when $l(g_1) = l(\gamma_1)$ and $l(g_2) = \sum_{i = 2}^k l(\gamma_i)$, which can only occur when $\gamma$ is the V-shaped geodesic, as the V-shaped geodesic contains the unique shortest paths from $M$ to the given edge of $X_3$.

On $X_5$, the above proof can be modified if we can show that $\gamma_1 \rightarrow \ldots \rightarrow \gamma_k$ hits the same edge $e_2$ as $g_1$ does, $n = 2$ skip numbers away from $M$. Assuming this edge is hit between $\gamma_j$ and $\gamma_{j + 1}$, (\ref{eqn:vshapepf}) will then follow in the same way with $l(g_1) \leq \sum_{i = 1}^j l(\gamma_i)$ and $l(g_2) \leq \sum_{i = j + 1}^k l(\gamma_i)$ since $\gamma_1 \rightarrow \ldots \rightarrow \gamma_j$ is a path from $M$ to the edge and $\gamma_{j + 1} \rightarrow \ldots \rightarrow \gamma_n$ is a path from the edge to $M$ and $g_1, g_2$ are perpendiculars. Note that it also suffices to show that the edge $e_3$, $3$ skip numbers away from $M$, is hit along $\gamma_1 \rightarrow \ldots \rightarrow \gamma_k$, since a path from $M$ to $e_3$ and back to $M$ is also no shorter than $l(g_3) + l(g_4) = l(g_1) + l(g_2)$ by the symmetry of $X_5$.

To show that $e_2$ or $e_3$ is hit along $\gamma_1 \rightarrow \ldots \rightarrow \gamma_k$, we note that $s_1 \leq n = 2 \Rightarrow s_1 = 1$ or $2$. If $s_1 = 2$, $e_2$ is hit, and if $s_1 = 1$, the skip number $s_2$ of $\gamma_2$ is $1$ or $2$ by Proposition~\ref{skips}, implying that $\gamma_2$ hits $e_2$ or $e_3$.  Thus the desired result is attained. 

Note that for $X_3$ we have $L= \sqrt{3}diam$ and for $X_5$ we have $L \approx 3.1diam$. As $n$ grows the doubled odd-gons converge to the doubled disk, and we have $L \to 4diam $.

\subsection{V-shaped geodesics on $X_{2n+1}$}
Here we prove Theorem~\ref{vshapethm} in general, considering the cases where $n \geq 3$. Again we focus on proving (\ref{eqn:vshapepf}).

The intuition for the proof is as follows: approximating $X_{2n + 1}$ with $X_\infty$, the path $g_1 \rightarrow g_2$ traverses the doubled disk along diameters. The path $\gamma_1 \rightarrow \gamma_2 \rightarrow \ldots \rightarrow \gamma_k$  starts and ends at the same point, so it winds around the doubled disk at least once. Thus, the intercepted arc angles sum to no less than $2\pi$. Given a counterclockwise orientation, we have two cases: all arcs are minor arcs or some arc is major. If all arcs are minor arcs, we may use the fact that the sine function, which determines chord lengths, is concave and increasing over $[0, \frac{\pi}{2}]$ to show that a traversal along the diameters is shortest. If some arc corresponding to $\gamma_i$ is major, we may assume without loss of generality (changing orientation if need be) that the arc corresponding to $\gamma_1$ is minor, and Proposition~\ref{skips} implies that there is some $1 < j < i$ such that $\gamma_j$ has a skip number close to that of a diameter, making traversal along $\gamma_j$ and back to the starting point of $\gamma_j$ ($\gamma_{j + 1} \rightarrow \ldots \rightarrow \gamma_{j  - 1}$) no shorter than traversal along two diameters $g_1, g_2$ by the triangle inequality.
%should this say "the arc corrsponding to gamma_1 is major" in the WLOG?

We now make this intuition more rigorous. The following paragraph describes a construction depicted in Figure~\ref{VSP2}.

\begin{figure}[ht]
\includegraphics[scale = 0.35]{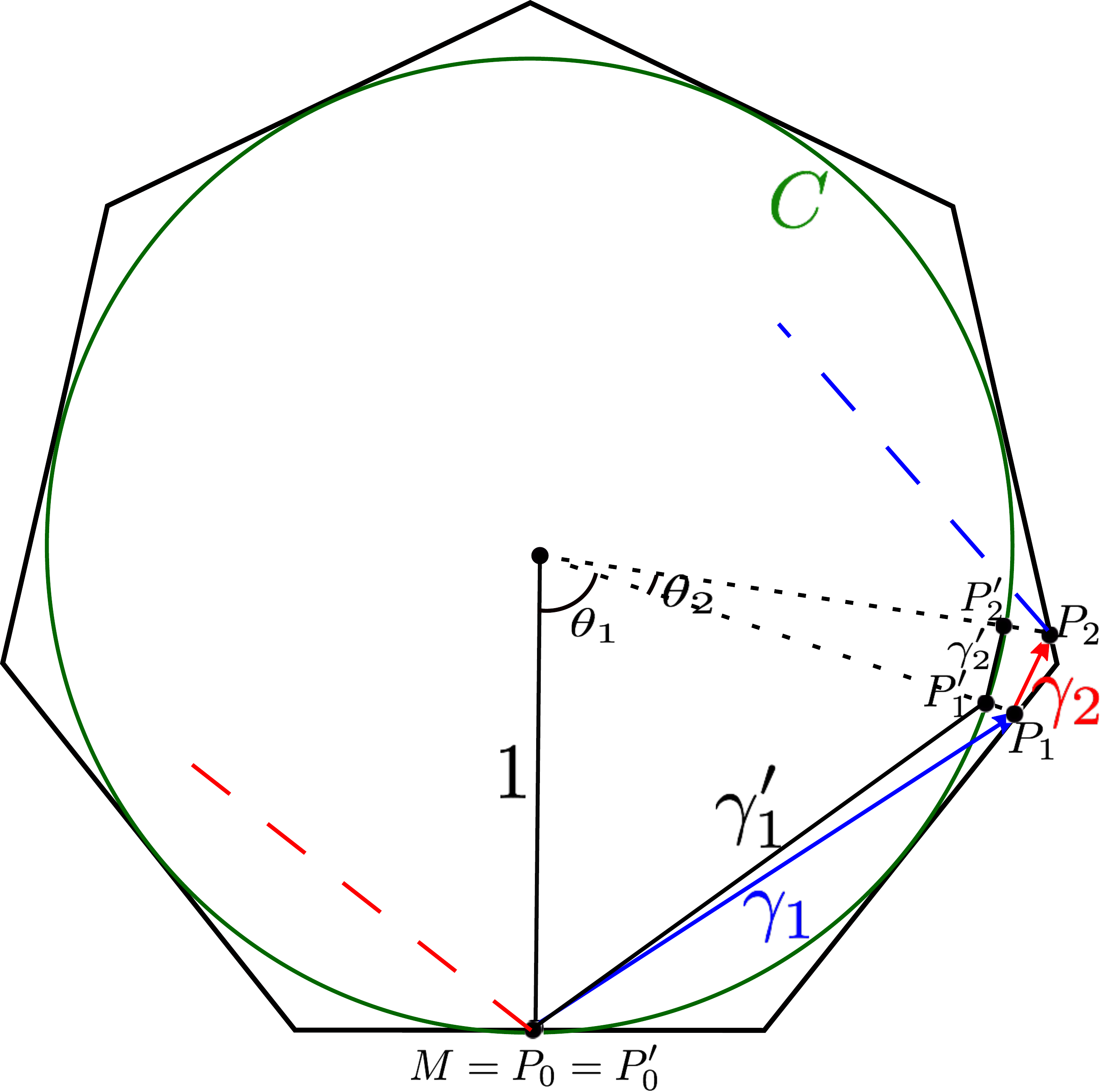}
\caption{Construction for the General Proof of Theorem~\ref{vshapethm}} 
\label{VSP2}
\end{figure}

Let $P_0, P_1, P_2, \ldots P_{k - 1}$ be the intersection points of the path $\gamma_1 \rightarrow \gamma_2 \rightarrow \ldots \rightarrow \gamma_k$ with the edges of $X_{2n + 1}$ in succession, where $P_0 = M$ and $\gamma_i$ extends from $P_{i - 1}$ to $P_i$, and let $C$ be the inscribed circle of $X_{2n + 1}$. Let $P_0^\prime, P_1^\prime, P_2^\prime, \ldots, P_{k - 1}^\prime$ lie on the circumference of $C$, so that $P_i^\prime$ is the intersection point of this circumference with the line between $P_i$ and the center of $C$.
%scaled towards the center of $C$. 
%"Scaled" should be a little more clear here -- by the diagram it's clear that it's on a straight line from P_i to the center of the circle, but it could also be on the line from P_i-1 to P_i, which seems equally valid as an interpretation.  I get that "towards the center of C" could make it clear but why not just name that point O and talk about explicit intersections of the two lines.
Let $\gamma_i^\prime$ be the segment connecting $P_{i - 1}^\prime$ to $P_i^\prime$, so that $\gamma_1^\prime \rightarrow \ldots \rightarrow \gamma_k^\prime$ is a path from $P_0 = M$ to $M$ winding around the circumference of $C$, given a counterclockwise orientation. Let $\theta_1, \theta_2, \ldots \theta_{k}$ be the measures of corresponding intercepted arcs. Figure~\ref{VSP2} shows this construction for $\gamma_1$ and $\gamma_2$. 

Then, fixing the radius of $C$ at $1$, we have the following:

\[
\displaystyle \sum_{i = 1}^k l(\gamma_i) \geq \displaystyle \sum_{i = 1}^k l(\gamma_i^\prime) = 2\sum_{i = 1}^k \sin \left(\frac{\theta_i}{2}\right).
\]

Because $\gamma_1^\prime \rightarrow \ldots \rightarrow \gamma_k^\prime$ loops around the circumference of $C$ from $M$ to $M$, $\displaystyle \sum_{i = 1}^k \theta_i = 2m\pi$ for some $m \in \mathbb{Z}^+$.

If all intercepted arcs are minor, each $\theta_i < \pi$ and we may use the fact that $\sin$ is concave and increasing over $[0, \frac{\pi}{2}]$, along with $\displaystyle \sum_{i = 1}^k \theta_i = 2m\pi \Rightarrow \displaystyle \sum_{i = 1}^k \frac{\theta_i }{2} = m\pi = 2m \cdot \frac{\pi}{2}$, to yield

\[
2\sum_{i = 1}^k \sin \left(\frac{\theta_i}{2}\right) \geq 2(2m)\sin \left(\frac{\pi}{2} \right) \geq 4\sin \left(\frac{\pi}{2} \right) = 4.
\]

Note that this lower bound of $4$ is twice the diameter of $C$. We will now show that $l(g_1) + l(g_2)$ is no greater than twice this diameter, proving the result for the case where all intercepted arcs are minor. 

If the radius of the inscribed circle $C$ is $1$, we may calculate that (see Figure~\ref{VSP3}) \[l(g_1) = \sin \left(\theta\right) \left(1 + \csc\left(\theta\right)\right), \] 

where $\theta = \frac{\pi (2n - 1)}{2(2n + 1)}$.

\begin{figure}[ht]
\includegraphics[scale = 0.35]{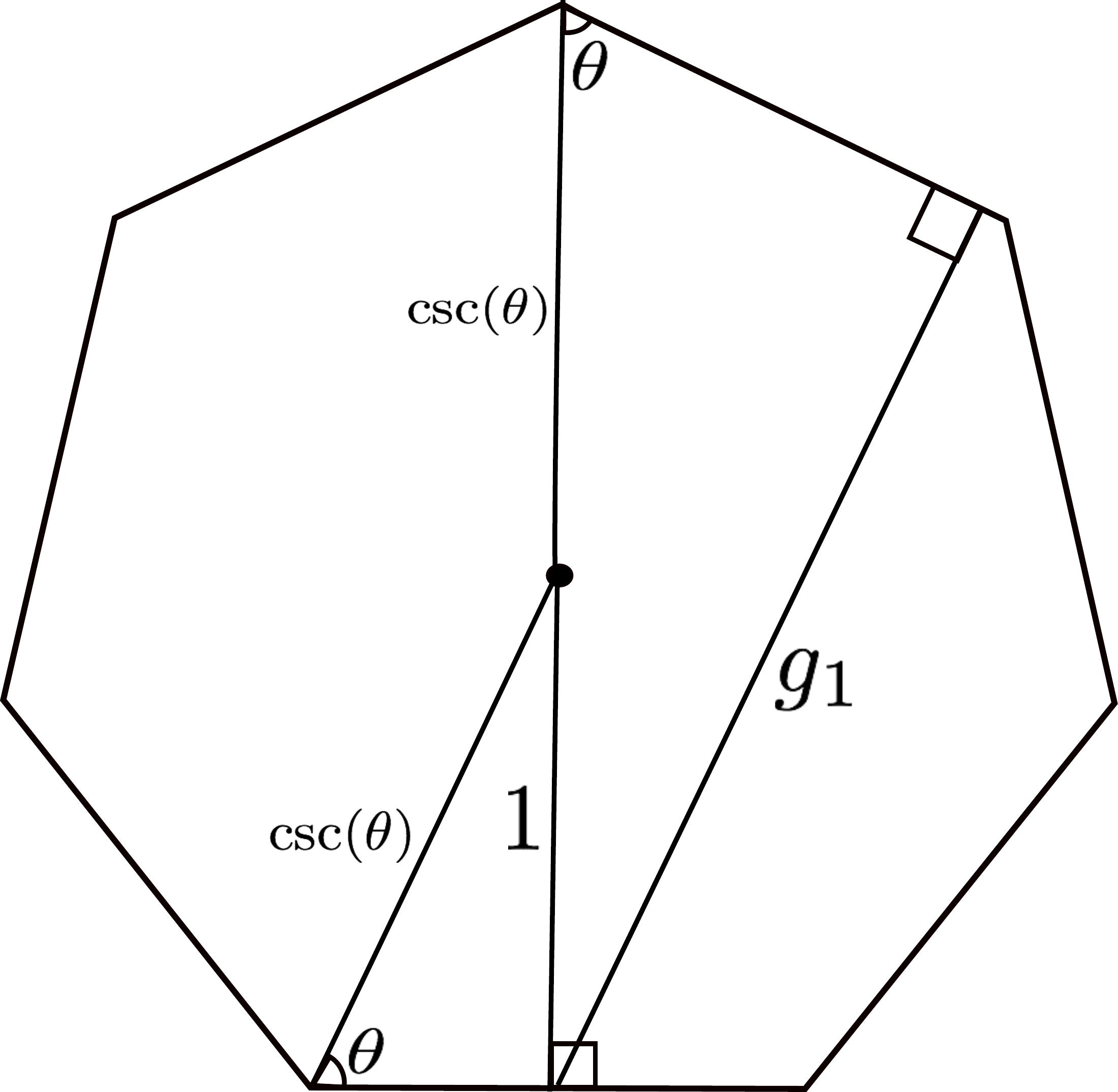}
\caption{Calculation of $l(g_1)$} 
\label{VSP3}
\end{figure}

Then, since $l(g_1) = l(g_2)$, we have
\[
 l(g_1) + l(g_2) =  2\sin\left(\theta\right) \left(1 + \csc \left(\theta\right)\right) = 
 \]
 
 \[=
 2\sin\left(\theta\right) + 2 < 2 + 2 = 4
 \]
 
as desired. 

%(Note the strict inequality here, as $\theta$ is strictly less than $\frac{\pi}{2}$ for all values of $n$.)

Now we consider the case where some intercepted arc is major, so that some $\theta_i \geq \pi$. In this case, there is some skip number $s_i \geq \lfloor \frac{2n + 1}{2} \rfloor = n$. Because $s_1 \leq n$, Proposition~\ref{skips} implies that there is $1 \leq j \leq i$ with $s_j = n$. Let $j$ be the smallest such $j$ so that $s_j$ is the first occurrence of skip number $n$. If the corresponding $l(\gamma_j) \geq l(g_1)$, by the triangle inequality, $\displaystyle \sum_{i= 1,  i \neq j}^k l(\gamma_i) \geq l(\gamma_j) \geq l(g_1) = l(g_2)$, since $\gamma_{j + 1} \rightarrow \ldots \gamma_k \rightarrow \gamma_1 \rightarrow \ldots \gamma_{j - 1}$ is a path connecting the endpoints of $\gamma_j$. Then $\displaystyle \sum_{i = 1}^k l(\gamma_i) = l(\gamma_j) + \displaystyle \sum_{i= 1,  i \neq j}^k l(\gamma_i) \geq l(g_1) + l(g_2)$. 

Note that we only have equality when $l(\gamma_j) = l(g_1)$ and $\displaystyle \sum_{i= 1,  i \neq j}^k l(\gamma_i) = l(g_2)$.  Since $\gamma_j$ and the other first $k-1$ segments of $\gamma$ connect the same two points and $l(g_1) = l(g_2)$, we must have $k=2$ by the triangle inequality.  Furthremore, $\gamma_1$ and $\gamma_2$ must be the same segment.  Since $k=2$, one of these segments must hit $M$, and for the geodesic to follow the same segment twice, it must hit the other edge perpendicularly.  (Note that $\gamma_1$ cannot hit $M$ perpendicularly, or else it would pass through a vertex.)  But the only geodesic satisfying all of these restrictions is the V-shaped geodesic.  Thus are inequality is strict unless $\gamma$ is the V-shaped geodesic. 

Now, in the case where $l(\gamma_j) < l(g_1)$, we will show that there is a shorter path from $M$ to $M$ along the circumference of $C$ with only minor arcs intercepted, reducing the problem to the proven case above. First, if $l(\gamma_j) < l(g_1)$, $\theta_j < \pi$, and we have a minor arc intercepted. This is because the arc resulting from $g_1$ is a minor, and if we denote this arc measure as $\alpha$, $l(\gamma_j) < l(g_1) \Rightarrow \theta_j < \alpha < \pi$, using the fact that the skip number of $\gamma_j$ is the same as that of $g_1$ (see Figure~\ref{VSP4}).
%I kind of understand what you're trying to say but this sentence is incredibly confusing and the diagram isn't much help.

\begin{figure}[ht]
\includegraphics[scale = 0.35]{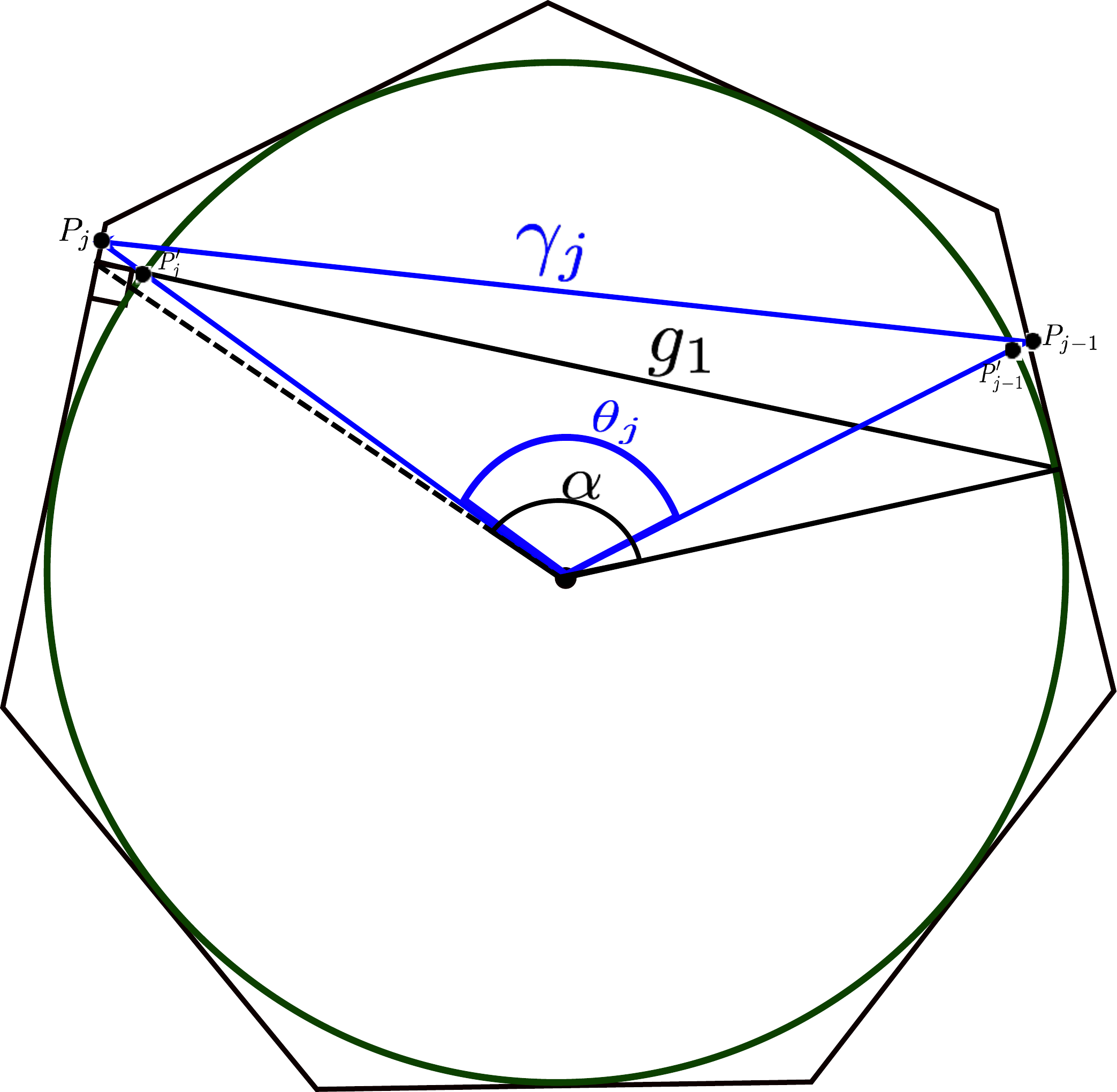}
\caption{$l(\gamma_j) < l(g_1) \Rightarrow  \theta_j < \alpha < \pi$} 
\label{VSP4}
\end{figure}

Since $s_j = n$ and since $g_1$ is the shortest path from the midpoint $M$ of an edge to the edge $n$ skip numbers away, $l(\gamma_j) < l(g_1) \Rightarrow j \neq 1$. Then $j > 1$, and using $s_j = n$ and $s_1 \leq n$, Proposition~\ref{skips} implies $s_{j - 1} = n - 1$, since $s_j$ is the first occurrence of skip number $n$. Thus, $\gamma_{j - 1}^\prime$ also intercepts a minor arc. We add one more chord intercepting a minor arc to the collection $\{\gamma_j^\prime, \gamma_{j - 1}^\prime\}$ as follows: $s_j + s_{j - 1} = 2n - 1$, so there are $2n + 1 - (2n - 1) = 2$ skip numbers between $P_j$, the end point of $\gamma_j$, and $P_{j -2}$, the start point of $\gamma_{j - 1}$. Thus, a line segment $s$ from $P_j^\prime$ to $P_{j - 2}^\prime$ intercepts a minor arc of $C$, since for the $n \geq 3$ case, $2 < n = \lfloor \frac{2n + 1}{2} \rfloor$.

%Can add figure here to illustrate above if confusing.

%(see Figure~\ref{VSP5}). 

%\begin{figure}[ht]
%\includegraphics[scale = 0.25]{VshapeProof1.pdf}
%\caption{Labeling Segments of V-Shaped Geodesic in a %Counterclockwise traversal} 
%\label{VSP5}
%\end{figure}

Since $\gamma_{j + 1}^\prime \rightarrow \ldots \rightarrow \gamma_k^\prime \rightarrow \gamma_1^\prime \rightarrow \ldots \rightarrow \gamma_{j - 2}^\prime$ is a path from $P_j^\prime$ to $P_{j - 2}^\prime$, by the triangle inequality $\displaystyle \sum_{i = 1, i \not\in \{j, j - 1\}}^k l(\gamma_i^\prime) \geq l(s)$. Now, with the path of chords $\gamma_{j}^\prime \rightarrow \gamma_{j - 1}^\prime \rightarrow s$ along the circumference of $C$, we have reduced to the case with only minor arcs intercepted. Thus we have

\[
\sum_{i = 1}^k l(\gamma_i^\prime)  =  l(\gamma_j^\prime) + l(\gamma_{j - 1}^\prime)  + \sum_{i = 1, i \not\in \{j, j - 1\}}^k l(\gamma_i^\prime) \geq  l(\gamma_j^\prime) + l(\gamma_{j - 1}^\prime) + l(s) > l(g_1) + l(g_2)
\]

as desired. 

Note that all the arguments in this section can be easily adopted to show that the $V$-shaped geodesic is uniquely the shortest on $X_{2n + 1}$.

%by the previous analysis with only minor arcs since $\gamma_{j - 1}^\prime \rightarrow \gamma_j^\prime \rightarrow s$ is a path of chords around the circumference of $C$ with only minor arcs intercepted.

\subsection{Bounds in terms of area}
One can also bound the length of the shortest closed geodesic on a Riemannian 2-sphere in terms of the area. The best known bound is $L\leq 4 \sqrt{2} \sqrt{area}$ and is due to Rotman \cite{rotman2}. It is conjectured by Calabi and Croke that $L \leq \sqrt{12} \sqrt{area}$. We note that for the singular space $X_3$ we have $L\approx 1.9 \sqrt{area}$ and for $X_5$ we have $L\approx 2.7 \sqrt{area}$. As $n$ grows the doubled odd-gons converge to the doubled disk, and $L \to \frac{4\sqrt{2}}{\sqrt{\pi}} \sqrt{area} \approx 3.2 \sqrt{area}$.


\begin{thebibliography}{2}
\bibitem{ade}
I.~Adelstein. 
Existence and non-existence of half-geodesics on $S^2$. 
\textit{Proc. Amer. Math. Soc.}, 144(7): 3085-3091, 2016.

\bibitem{ade2}
I.~Adelstein.
Minimizing closed geodesics via critical points of the uniform energy. 
\textit{Math. Res. Lett.}, 23(4): 953-972, 2016.

\bibitem{epstein}
I.~Adelstein, J.~Epstein.
Morse theory for the uniform energy.
\textit{J.~Geometry}, 108(3): 1193-1205, 2019.

\bibitem{fong}
I.~Adelstein, A.~Fong.
Closed geodesics on doubled polygons.
\textit{Preprint}, 2019.

\bibitem{croke}
F. Balacheff, C. Croke, M. Katz.
A Zoll counterexample to a geodesic length conjecture. 
\textit{Geom. and Funct. Analysis}, 19: 1–10, 2009.

\bibitem{davis}
Diana Davis, Dmitry Fuchs, Sergei Tabachnikov. Periodic trajectories in the regular pentagon, Moscow Mathematical Journal, Volume 11, Issue 3 (2011), 439-461.

\bibitem{fuchs}
D.~Fuchs. 
Geodesics on a regular dodecahedron.
\textit{Preprint}, Max Planck Institute for Mathematics. 2009.

\bibitem{wkh}W.~K.~Ho.
Manifolds without 1/k-geodesics.
\textit{Israel J. Math.}, 168(1): 189-200, 2008.

\bibitem{stack}
J.~Marie. 
Length of two sides in a quadrilateral with given angles.
\emph{Mathematics Stack Exchange}. https://math.stackexchange.com/questions/2239061.

\bibitem{rotman}
A.~Nabutovsky, R.~Rotman.
The length of the shortest closed geodesic on a two-dimensional sphere.
\textit{Int. Math. Res. Not.}, 23: 1211–1222, 2002.

\bibitem{rotman2}
R.~Rotman.
The length of a shortest closed geodesic and the area of a 2-dimensional sphere.
\textit{Proc. Amer. Math. Soc.}, 134(10): 3041-3047.

\bibitem{sab}
S. Sabourau.
Filling radius and short closed geodesics of the sphere.
\textit{Bull. Soc. Math. France} 132(1): 105–136, 2004.

\bibitem{sor}
C.~Sormani. 
Convergence and the length spectrum. 
\textit{Adv. Math.}, 213(1): 405-439, 2007.

\bibitem{schwartz}
R.~Schwartz.
Obtuse Triangular Billiards I: Near the $(2, 3, 6)$ Triangle.
\textit{Experimental Math.} 15(2): 161-182, 2006.

\bibitem{veech}
W.~A.~Veech. 
The billiard in a regular polygon. 
\textit{Geom. Funct. Anal.}, 2(3): 341-379, 1992. 

\end{thebibliography}
\end{document}